\documentclass{article}

\usepackage{amsfonts, amsmath, amssymb, amsthm, amscd}
\usepackage{ascmac}
\usepackage[dvipdfm]{color}
\usepackage{verbatim}
\usepackage[dvips]{graphicx}
\usepackage{graphics}
\usepackage[all,2cell]{xypic}

\objectmargin+{1mm}
\labelmargin+{0.8mm}
\SelectTips{cm}{12}

\UseAllTwocells

\theoremstyle{plain}  
\newtheorem{thm}{Theorem}[section]

\theoremstyle{definition}

\newtheorem*{conv}{Conventions}
\newtheorem{para}[thm]{}

\theoremstyle{remark}
\newtheorem*{claim}{claim}

\DeclareMathOperator{\cA}{\mathcal{A}}
\DeclareMathOperator{\cB}{\mathcal{B}}
\DeclareMathOperator{\cC}{\mathcal{C}}
\DeclareMathOperator{\calD}{\mathcal{D}}
\DeclareMathOperator{\cE}{\mathcal{E}}

\DeclareMathOperator{\cI}{\mathcal{I}}
\DeclareMathOperator{\cJ}{\mathcal{J}}

\DeclareMathOperator{\calL}{\mathcal{L}}
\DeclareMathOperator{\cM}{\mathcal{M}}

\DeclareMathOperator{\cS}{\mathcal{S}}
\DeclareMathOperator{\cT}{\mathcal{T}}
\DeclareMathOperator{\cU}{\mathcal{U}}
\DeclareMathOperator{\cV}{\mathcal{V}}

\DeclareMathOperator{\cX}{\mathcal{X}}
\DeclareMathOperator{\cY}{\mathcal{Y}}

\DeclareMathOperator{\bbZ}{\mathbb{Z}}

\DeclareMathOperator{\fS}{\mathfrak{S}}

\UseAllTwocells

\def\sn{\smallskip\noindent}
\def\mn{\medskip\noindent}

\newcommand{\cf}{\textrm{cf.}\;}
\newcommand{\Cof}{\operatorname{Cof}}
\newcommand{\coim}{\operatorname{Coim}}
\newcommand{\coker}{\operatorname{Coker}}
\newcommand{\colim}{\operatorname{colim}}

\newcommand{\Hom}{\operatorname{Hom}}

\newcommand{\id}{\operatorname{id}}

\newcommand{\Isom}{\operatorname{Isom}}
\newcommand{\isoto}{\overset{\scriptstyle{\sim}}{\to}}

\newcommand{\Ker}{\operatorname{Ker}}

\newcommand{\Mor}{\operatorname{Mor}}

\newcommand{\Ob}{\operatorname{Ob}}
\newcommand{\onto}[1]{\stackrel{#1}{\to}}

\newcommand{\rdef}{\twoheadrightarrow}
\newcommand{\rinc}{\hookrightarrow}
\newcommand{\rinf}{\rightarrowtail}

\newcommand{\Sp}{\operatorname{\bf Sp}}

\title{Fibration theorem for Waldhausen $K$-theory}
\date{}

\author{Satoshi Mochizuki}

\begin{document}

\maketitle

\begin{abstract}
The main result of this paper is a new flavour of 
Waldhausen's fibration theorem for Waldhausen $K$-theory 
under a different set of hypothesis. 
The theorem says that for a small category with 
cofinrations $\cC$ and certain classes of weak equivalences 
$v$ and $w$ in $\cC$, 
a sequence of simplicial categories:
$$vS_{\cdot}\cC^w\to vS_{\cdot}\cC \to wS_{\cdot}\cC$$
is a fibration sequence up to homotopy.
\end{abstract}

\section*{Introduction}

The main purpose of this short note is to give a variant of 
the generic fibration theorem for Waldhausen $K$-theory. 
The theorem is first proven in \cite{Wal85} and 
improved in \cite{Sch06}. 
To give more precise imformation, 
let $\cC$ be a small category with cofibrations in the sense of \cite{Wal85} 
and $v\subset w$ sets of weak equivalences in $\cC$ such that 
$w$ satisfies the extension axiom. 
Then the full subcategory $\cC^w$ of $\cC$ consisting of those of objects $x$ such that 
the canonical morphism $0\to x$ is in 
$w$ is a subcategory with cofibrations in $\cC$ and 
the inclusion functor $\cC^w\rinc \cC$ and 
the identity functor of $\cC$ induce a sequence of 
simplicial categories: 
$$vS_{\cdot}\cC^w\to vS_{\cdot}\cC \to wS_{\cdot}\cC.$$ 
The original theorem in \cite{Wal85} or \cite{Sch06} 
says that if $w$ is saturated and the pair $(\cC,w)$ satisfies 
the factorization axiom, then the sequence above is a fibration sequence up to homotopy. 
In this paper we give an another sufficient applicable condition 
which makes the sequence above a fibration sequence up to homotopy. 
(See Theorem~\ref{thm:fib thm}.) 
In sections $3$ and $4$, 
we will illustrate examples of fibration sequences of $K$-theory.

\begin{conv}
We mainly follow the notations in \cite{Wal85}. 
We say that a class $w$ of weak equivalences in a category of cofibrations 
is {\it extensional} 
if $w$ satisfies the extension axiom. 
For a pair of small categories $\cX$ and $\cY$, 
we write $\cY^{\cX}$ for the category whose objects are 
functors from $\cX$ to $\cY$ and whose morphisms are 
natural transformations. 
\end{conv}

\sn
\textbf{Acknowledgements.} 
The author wishes to express his deep gratitude 
to Marco Schlichting for stimulating discussions. 
He also very thanks to referees for 
carefully reading a preprint version 
of this paper and giving innumerable and valuable comments 
(for example Remark~\ref{rem:not cover}), 
to make the paper more readable.

\section{Sets of morphisms in a small category}
\label{sec:Class mor}

In this section, let $\cC$ and $\calD$ be small categories and 
we write $\Mor \cC$ for the set of all morphisms in $\cC$. 
We mainly study heritability of properties for sets of morphisms in $\cC$ 
by taking a right cofinal subset in \ref{lem:fp com}, 
a pull-back by a functor in \ref{lem:pullback} and 
simplicial constructions in \ref{lem:cC(-;U)}. 
We start by giving a glossary about properties for sets of morphisms.

\begin{para}
\label{df:multiplicative systems}
{\bf Definition.} 
Let $\cS$ be a set of morphisms in $\cC$ containing the 
identity morphisms of objects in $\cC$. 
We say that $\cS$ is a 
{\it multiplicative set} if 
$\cS$ is closed under finite compositions. 
We say that $\cS$ is {\it strictly multiplicative set} 
if all isomorphisms in $\cC$ are in $\cS$ and 
if $\cS$ is closed under finite compositions. 
Suppose that $\bullet \onto{f} \bullet \onto{g} \bullet$ 
is a composable sequence of morphisms in $\cC$. 
If whenever two of $f$, $g$ and $gf$ are in $\cS$, 
then the third morphism is also, 
we say $\cS$ is a {\it saturated set} or 
{\it satisfies the saturation axiom}. 
We regard a multiplicative set of $\cC$ as a subcategory of $\cC$. 
\end{para}

\begin{para}
\label{df:localizing set}
{\bf Definition.} 
Let $\cS$ and $\cT$ be sets of morphisms in $\cC$. 
We set 
$$\cT\circ\cS:=\{fg\in\Mor\cC;
\text{$\bullet\onto{g}\bullet\onto{f}\bullet$ 
where $g$ is in $\cS$ and $f$ is in $\cT$}\}.$$ 
We say that $\cS$ is {\it right permutative} 
with respect to $\cT$ 
if for any morphisms $a\colon x\to z$ in $\cS$ and 
$b\colon y \to z$ in $\cT$, 
there are an object $u$ in $\cC$ and morphisms $a'\colon u \to y$ in $\cS$ and 
$b'\colon u\to x$ in $\cT$ such that $ab'=ba'$. 
We say that $\cS$ is {\it right reversible} 
with respect to $\cT$ 
if for any morphisms $a$, $a'\colon x\to y$ in $\cT$, 
if there exists a morphism $b\colon y \to z$ in $\cS$ such that $ba=ba'$, 
then there exists a morphism $c\colon u\to x$ in $\cS$ such that $ac=a'c$. 
We say that $\cS$ is {\it right {\O}re} 
with respect to $\cT$ 
if $\cS$ is right permutative and right reversible 
with respect to $\cT$. 
We say that $\cS$ is {\it right localizing} ({\it in $\cC$}) 
if $\cS$ is a multiplicative and right {\O}re set with respect to 
$\Mor\cC$. 
We say that $\cT$ is {\it right cofinal} in $\cS$ if 
$\cT\subset \cS$ and 
for any morphism $x\to y$ in $\cS$, there is a morphism 
$z \to x$ in $\cT$ such that the composition 
$z\to y$ is also in $\cT$.
\end{para}

\begin{para}
\label{ex:class of isomorphisms}
{\bf Example. (Set of all isomorphisms).} 
We write $i_{\cC}$ or shorty $i$ for the class of all isomorphisms in $\cC$. 
Then $i_{\cC}$ is a saturated, strictly multiplicative, right localizing set 
in $\cC$.
\end{para}

\begin{para}
\label{lem:fp com}
{\bf Lemma.} 
{\it
Let $\cT\subset\cS$ and $\cU$ be sets of morphisms in $\cC$. 
Assume that $\cT$ is right cofinal in $\cS$.\\
$\mathrm{(1)}$ 
If $\cS$ is right permutative with respect to $\cU$ and 
$\cU\circ\cT\subset \cU$, 
then $\cT$ is also right permutative with respect to $\cU$.\\
$\mathrm{(2)}$ 
If $\cS$ is right reversible with respect to $\cU$, 
then $\cT$ is also right reversible with respect to $\cU$.\\
$\mathrm{(3)}$ 
If $\cT$ is right permutative with respect to $\cU$ and 
$\cT\circ\cU \subset \cU$, 
then $\cS$ is also right permutative with respect to $\cU$.\\
$\mathrm{(4)}$ 
Assume that $\cS$ is saturated and right localizing in $\cC$ and 
$\cT$ is a multiplicative set. 
Then the inclusion functor 
$j\colon\cT \rinc \cS$ is a homotopy equivalence.
}
\end{para}

\begin{proof}
$\mathrm{(1)}$ 
In order to show that $\cT$ is right permutative 
with respect to $\cU$, 
we must produce the dotted morphisms 
$s'''\colon e\to a$ in $\cT$ and 
$f''\colon e\to b$ in $\cU$
in the commutative diagram below with $f \in \cU$ and 
$s\in \cT$:
$$
\xymatrix{
e \ar@{-->}[r]^{f''\in \cU} \ar@{-->}[d]_{s'''\in \cT} & b \ar[d]^{s\in \cT}\\
a \ar[r]_{f\in \cU} & c.
}$$
Then by assumption, 
there are morphisms $f'\colon d\to b$ in $\cU$ and 
$s'\colon d \to a$ in $\cS$ such that $fs'=sf'$. 
Since $\cT$ is right cofinal in $\cS$, 
there is a morphism $s''\colon e\to d$ in $\cT$ such that 
the composition $s's''\colon e\to a$ is in $\cT$. 
By the condition $\cU\circ\cT\subset\cU$, 
the morphism $f's''$ is in $\cU$. 
We shall set $s''':=s's''$ and $f'':=f's''$. 

\sn
$\mathrm{(2)}$ 
In order to show that 
$\cT$ is right reversible with respect to $\cU$, 
for any morphisms $f$, $g\colon a\to b$ in $\cU$ 
such that there is a morphism 
$s\colon b\to c$ in $\cT$ such that $sf=sg$, 
we must produce a morphism $s'''\colon e\to a$ in $\cT$ 
such that $fs'''=gs'''$. 
Then there is a morphism $s'\colon d \to a$ in $\cS$ such that $fs'=gs'$. 
Since $\cT$ is right cofinal in $\cS$, there is a morphism $s''\colon e\to d$ 
such that the composition $s's''\colon e\to a$ is in $\cT$. 
We shall set $s''':=s's''$.

\sn
$\mathrm{(3)}$ 
In order to show that 
$\cS$ is right permutative with respect to $\cU$, 
we must produce the dotted morphisms 
$s''\colon d\to a$ in $\cS$ and 
$f''\colon d\to b$ in $\cU$ in the commutative diagram below 
with $f\in \cU$ and $s\in \cS$: 
$$
\xymatrix{
d \ar@{-->}[r]^{f''\in\cU} \ar@{-->}[d]_{s''\in \cT} & b \ar[d]^{s\in \cS}\\
a \ar[r]_{f\in \cU} & c.
}$$
Then there is a morphism $s'\colon b'\to b$ in $\cT$ such that the composition 
$ss'\colon b'\to c$ is in $\cT$ by assumption. 
Then there are morphisms $s''\colon d\to a$ in $\cT$ and 
$f'\colon d\to b'$ in $\cU$ such that $ss'f'=fs''$. 
By assumption the composition $s'f'\colon d\to b$ is in $\cU$. 
We shall set $f'':=s'f'$. 

\sn
$\mathrm{(4)}$ 
Let $x$ be an object in $\cS$. 
We write $j/x$ for the category whose object is a pair $(y,a)$ of 
an object $y$ in $\cT$ and a morphism $a\colon y\to x$ in $\cS$ and 
whose morphism $\alpha\colon (y,a)\to(z,b)$ 
is a morphism $\alpha\colon y\to z$ in $\cT$ 
such that $a=b\alpha$. 
Since the object $(x,\id_x)$ is in $j/x$, 
the category $j/x$ is a non-empty category.

\begin{claim}
$j/x$ is a cofiltering category. Namely\\
$\mathrm{(a)}$ 
For any objects $(y,a)$ and $(z,b)$ in $j/x$, 
there are morphisms $\alpha\colon (w,c)\to (y,a)$ and 
$\beta\colon (w,c)\to (z,b)$ in $j/x$.\\
$\mathrm{(b)}$ 
For any morphisms $\alpha$, $\beta\colon (y,a)\to (z,b)$ in $j/x$, 
there is a morphism $\gamma\colon (w,c)\to (y,a)$ such that 
$\alpha\gamma=\beta\gamma$.
\end{claim}

\begin{proof}[Proof of claim]
$\mathrm{(a)}$ 
Since $\cS$ is right permutative with respect to $\Mor\cC$, 
there are morphisms $b'\colon w'\to y$ and $a'\colon w'\to z$ such that 
$a'$ is in $S$. 
Then by the saturated axiom for $\cS$, 
$ba'=ab'$ and $b'$ are also in $\cS$. 
By right cofinality of $\cT$ in $\cS$, 
there is a morphism $t\colon w''\to w'$ such that 
the composition $a't$ is in $\cT$. 
By right cofinality of $\cT$ in $\cS$ again, 
there is a morphism $t'\colon w\to w''$ such that 
the composition $b'tt'$ is in $\cT$. 
Then we set $\alpha:= b'tt''$, $\beta: =a'tt'$ and $c:=a\alpha=b\beta$. 

\sn
$\mathrm{(b)}$ 
Since $\cS$ is right reversible with respect to $\Mor\cC$, 
the equalities $b\alpha=a=b\beta$ implies that 
there is a morphism $t\colon w'\to y$ in $\cS$ such that $\alpha t=\beta t$. 
By right cofinality of $\cT$ in $\cS$, 
there is a morphism $t'\colon w\to w'$ in $\cT$ 
such that the composition $tt'$ is in $\cT$. 
We set $\gamma:=tt'$ and $c:=a\gamma$.
\end{proof}

\sn
By Corollary~2 and Theorem~A in \cite[\S 1]{Qui73}, we obtain the desired result.
\end{proof}

\begin{para}
\label{df:pull back}
{\bf Definition.} 
Let $\phi\colon\cC\to\calD$ be a functor and 
$\cS$ a non-empty set of morphisms in $\calD$. 
We define the set of morphisms $\phi^{-1}\cS$ in $\cC$ 
the {\it pull-back of $\cS$ by $\phi$} by the formula
$$\phi^{-1}\cS\colon =\{f\in\Mor\cC;\phi(f)\in\cS\}.$$ 
\end{para}

\begin{para}
\label{lem:pullback}
{\bf Lemma.}
{\it
Let $\phi\colon\cC\to\calD$ be a functor and 
$\cS$ a non-empty set of morphims in $\calD$. 
Then\\
$\mathrm{(1)}$ 
If $\cS$ is a multiplicative in $\cC$, then $\phi^{-1}\cS$ is also. 
If $\cS$ is a strictly multiplicative in $\cC$, then $\phi^{-1}\cS$ is also. 
If $\cS$ is a saturated set in $\cC$, 
then $\phi^{-1}\cS$ is also.\\
$\mathrm{(2)}$ 
If $\phi$ is full and essentially surjective and if 
$\cT$ is right cofinal in $\cS$, 
then $\phi^{-1}\cT$ is right cofinal in $\phi^{-1}\cS$.\\
$\mathrm{(3)}$ 
If $\phi$ is an equivalence of categories, 
$\cS$ and $\cT$ are strictly multiplicative sets 
and $\cS$ is right permutative 
with respect to $\cT$, 
then $\phi^{-1}\cS$ is right permutative 
with respect to $\phi^{-1}\cT$.\\
$\mathrm{(4)}$ 
If $\phi$ is an equivalence of categories, 
$\cS$ is a strictly multiplicative and right reversible set 
with resect to $\cT$, 
then $\phi^{-1}\cS$ is right reversible 
with respect to $\phi^{-1}\cT$.
}
\end{para}

\begin{proof}
$\mathrm{(1)}$ 
Since a functor sends an identity morphism 
to an identity morphism and 
an isomorphism to an isomorphism, 
if $\cS$ is closed under identity morphisms, 
then $\phi^{-1}\cS$ is also and 
if $\cS$ is closed under isomorphisms, 
then $\phi^{-1}\cS$ is also. 
Let $x\onto{a}y\onto{b}z$ be a pair of composable morphisms in $\cC$. 
If two of $ba$, $b$ and $a$ are in $\phi^{-1}\cS$, 
then two of $\phi(b)\phi(a)$, $\phi(b)$ and $\phi(a)$ are in $\cS$. 
Therefore if $\cS$ is closed under compositions, 
then $\phi^{-1}\cS$ is also and 
if $\cS$ is a saturated set, then $\phi^{-1}\cS$ is also.

\sn
$\mathrm{(2)}$ 
In order to show that 
$\phi^{-1}\cT$ is right cofinal in $\phi^{-1}\cS$, 
for any morphism $s\colon x\to y$ 
in $\phi^{-1}\cS$, 
we must produce a morphism 
$t''\colon z\to x$ in $\phi^{-1}\cT$ such that 
$st''$ is also in $\phi^{-1}\cT$. 
Then by assumption, 
there is a morphism $t\colon z'\to \phi(x)$ in $\cT$ such that the composition 
$\phi(s)t$ is in $\cT$. 
By essential surjectivity of $\phi$, 
there are an object $z$ in $\cC$ and an isomorphism 
$t'\colon \phi(z)\isoto z'$. 
BY fullness of $\phi$, 
there is a morphism $t''\colon z\to x$ in $\cC$ such that $\phi(t'')=tt'$. 
Since $\cT$ is a strictly multiplicative set, 
the composition $tt'$ is in $\cT$ and 
therefore $t''$ and the composition $st''$ are in $\phi^{-1}\cT$. 

\sn
$\mathrm{(3)}$ 
In order to show that 
$\phi^{-1}\cS$ is right permutative with respect to $\phi^{-1}\cT$, 
we must produce the dotted morphisms 
$s''\colon w\to y$ in $\phi^{-1}\cS$ and 
$t''\colon w\to x$ in $\phi^{-1}\cT$ 
in the commutative diagram below with 
$s\in\phi^{-1}\cS$ and $t\in\phi^{-1}\cT$: 
$$
\xymatrix{
w \ar@{-->}[r]^{s''\in\phi^{-1}\cS} \ar@{-->}[d]_{t''\in\phi^{-1}\cT} & 
y \ar[d]^{t\in \phi^{-1}\cT}\\
x \ar[r]_{s\in \phi^{-1}\cS} & z.
}$$
Then there are morphisms $s'\colon w'\to \phi(y)$ in $\cS$ 
and $t'\colon w'\to \phi(x)$ 
in $\cT$ such that $\phi(s)t'=\phi(t)s'$. 
By essential surjectivity of $\phi$, 
there are an object $w$ in $\cC$ and an isomorphism $u\colon\phi(w)\isoto w'$ in $\calD$. 
Since $\cS$ and $\cT$ are strictly multiplicative sets, 
$s'u$ and $t'u$ are in $\cS$ and $\cT$ respectively. 
By fullness of $\phi$, 
there are morphisms $s''\colon w\to y$ and $t''\colon w\to x$ in $\cC$ such that 
$s'u=\phi(s'')$ and $t'u=\phi(t'')$. 
Notice that $s''$ and $t''$ are in 
$\phi^{-1}\cS$ and $\phi^{-1}\cT$ respectively. 
The equality 
$\phi(st'')=\phi(s)t'u=\phi(t)s'u=\phi(ts'')$ 
and faithfulness of $\phi$ imply the equality $st''=ts''$. 

\sn
$\mathrm{(4)}$ 
In order to show that 
$\phi^{-1}\cS$ is right reversible with respect to $\phi^{-1}\cT$, 
for any morphisms 
$a$, $b\colon x\to y$ in $\phi^{-1}\cT$ 
such that there is a morphism 
$s\colon y\to z$ in $\phi^{-1}\cS$ such that $sa=sb$, 
we must produce a morphism $t'\colon w\to x$ in $\phi^{-1}\cS$ 
such that $at'=bt'$. 
Then there is a morphism $t\colon w'\to \phi(x)$ in $\cS$ such that $at=bt$. 
By essential surjectivity of $\phi$, 
there are an object in $\cC$ and 
an isomorphism $u\colon\phi(w)\isoto w'$ in $\calD$. 
Since $\cS$ is a strictly multiplicative set, $tu$ is in $\cS$. 
By fullness of $\phi$, 
there is a morphism $t'\colon w\to x$ in $\cC$ such that $\phi(t')=tu$. 
Notice that $t'$ is in $\phi^{-1}\cS$. 
The equalities $\phi(at')=\phi(a)tu=\phi(b)tu=\phi(bt')$ and 
faithfulness of $\phi$ imply the equality $at'=bt'$.
\end{proof}

\begin{para}
\label{df:cC(-,S)}
{\bf Definition.} 
For a multiplicative set $\cS$ of $\cC$, 
we define $\cC(-,\cS)$ 
to be a simplicial subcategory of $\cC^{[-]}$ 
by sending a totally ordered set $[m]$ to $\cC(m,\cS)$ 
where $\cC(m,\cS)$ is the full subcategory of $\cC^{[m]}$ 
consisting of those functors 
$x\colon [m] \to \cC$ such that 
for any morphism $i\leq j$ in $[m]$, 
$x(i\leq j)$ is in $\cS$. 
For each $m$, 
we denote an object $x_{\cdot}$ in $\cC(m,\cS)$ by 
$$x_{\cdot}\colon x_0 \onto{i^x_0} x_1 \onto{i^x_1} x_2 \onto{i^x_2}\cdots 
\onto{i^x_{m-1}} x_m.$$
For a set $\cT$ of morphisms in $\cC$, 
we define $\cT\cC(m,\cS)$ 
to be the set of morphisms in $\cC(m,\cS)$ by the formula
$$\cT\cC(m,\cS)\colon =\{f\in\Mor\cC(m,\cS);\text{$f_i$ is in $\cT$ for any $0\leq i\leq m$}\}.$$
\end{para}

\begin{para}
\label{lem:cC(-;U)}
{\bf Lemma.} 
{\it
Let $\cS$, $\cT$, $\cU$ and $\cV$ be non-empty sets of morphisms in $\cC$ 
and $n$ a non-negative integer. 
Then\\
$\mathrm{(1)}$ 
Assume that 
$\cT$ is a multiplicative set and right cofinal in $\cS$, 
$\cS$ is right permutative with respect to $\cU$ and 
$\cU\circ\cT\subset\cT\circ\cU$. 
Then $\cT\cC(n,\cU)$ is right cofinal in $\cS\cC(n,\cU)$.\\
$\mathrm{(2)}$ 
Assume that $\cT\circ\cS\subset\cT$ and $\cU\circ\cS\subset\cS\circ\cU$, 
$\cS$ is right permutative with respect to both $\cT$ and $\cU$ and 
all morphisms in $\cS$ are monomorphisms. 
Then $\cS\cC(n,\cU)$ is right permutative with respect to $\cT\cC(n,\cU)$.\\
$\mathrm{(3)}$ 
Assume that 
$\cV$ is a multiplicative, right cofinal set in $\cS$, 
$\cU\circ\cV\subset\cV\circ\cU$ and 
$\cS$ is a multiplicative, right permutative with respect to $\cU$ 
and reversible set with respect to $\cT$. 
Then $\cS\cC(n,\cU)$ is right reversible with respect to $\cT\cC(n,\cU)$.
}
\end{para}

\begin{proof}
We proceed by induction on $n$. 
If $n=0$, then the assertion is hypothesis. 
We assume that $n\geq 1$. 
Let $\iota\colon [n-1]\rinc [n]$ be the inclusion functor. 
We write $\iota^{\ast}\colon\cC(n,\cU)\to\cC(n-1,\cC)$ 
for the induced functor by the composition with $\iota$. 

\sn
$\mathrm{(1)}$ 
In order to show that 
$\cT\cC(n,\cU)$ is right cofinal in $\cS\cC(n,\cU)$, 
for any morphism $s\colon a\to b$ in $\cS\cC(n,\cU)$, 
we must produce a morphism 
$t\colon c\to a$ in $\cT\cC(n,\cU)$ such that 
$st$ is also in $\cT\cC(n,\cU)$. 
Then by assumption, 
there is a morphism $t'\colon c'\to a_n$ in $\cT$ such that 
$s_nt'\colon c'\to b_n$ is in $\cT$. 
Since $\cT$ is right cofinal in $\cS$, 
the morphism $t'\colon c'\to a_n$ is in $\cS$. 
We construct the dotted morphisms 
$t_k''\colon c_k'' \to a_k$ $(0\leq k\leq n-1)$ in $\cS$ 
in the commutative diagram below 
by using the assumption that $\cS$ is right permutative with 
respect to $\cU$ and descending induction on $k$. 
$$
\xymatrix{
c_0'' \ar[r]^{i_0^{c''}} \ar@{-->}[d]_{t_0''} & 
c_1'' \ar[r]^{i_1^{c''}} \ar@{-->}[d]_{t_1''} & 
\cdots \ar[r]^{i_{n-2}^{c''}} & 
c_{n-1}'' \ar[r]^{i_{n-1}^{c''}} \ar@{-->}[d]_{t_{n-1}''} & 
c_n''=c' \ar[d]^{t_n''=t'}\\
a_0 \ar[r]_{i_0^a} & a_1 \ar[r]_{i_1^a} & \cdots \ar[r]_{i_{n-2}^a} & 
a_{n-1} \ar[r]_{i_{n-1}^a} & a_n.
}
$$
Hence there is the morphism 
$t''\colon c''\to a$ in $\cS\cC(n,\cU)$ 
such that $c_n''=c'$ and $t_n''=t'\colon c'\to a_n$. 
By applying the inductive hypothesis to 
$\iota^{\ast}(st'')\colon \iota^{\ast}c''\to \iota^{\ast}b$, 
there is a morphism $t^{(3)}\colon c^{(3)}\to \iota^{\ast}c$ 
in $\cT\cC(n-1,\cU)$ 
such that the composition 
$(\iota^{\ast}(st''))t^{(3)}$ is also in $\cT\cC(n-1,\cU)$. 
Then by the condition $\cU\circ\cT\subset\cT\circ\cU$, 
there are the dotted morphisms 
$j\colon c_{n-1}\to c^{(4)}$ in $\cU$ and 
$t^{(4)}\colon c^{(4)}\to c'=c_n''$ in $\cT$ 
in the commutative diagram below:
$$\xymatrix{
c_{n-1} \ar@{-->}[r]^{j\in \cU} \ar@{-->}[d]_{t_{n-1}^{(3)}\in\cT} & 
c^{(4)} \ar[d]^{t^{(4)}\in\cT}\\
c''_{n-1} \ar[r]_{i_{n-1}^{c''}\in\cU} & c'=c''_n.
}$$
We define $t\colon c\to a$ to be a morphism in $\cT\cC(n;\cU)$ by 
setting as follows.
$$
c_k:=
\begin{cases}
c_k^{(3)} & \text{if $0\leq k\leq n-1 $}\\
c^{(4)} & \text{if $k=n$}
\end{cases}, 
i_k^c:=
\begin{cases}
i_k^{(3)} & \text{if $0\leq k\leq n-2$}\\
j & \text{if $k=n-1$}
\end{cases},
t_k:=
\begin{cases}
t_k''t_k^{(3)} & \text{if $0\leq k\leq n-1$}\\
t_n''t^{(4)} & \text{if $k=n$}.
\end{cases}
$$
Then we can easily check that $st\colon c\to b$ is in $\cT\cC(n,\cU)$. 

\sn
$\mathrm{(2)}$ 
In order to show that 
$\cS\cC(n,\cU)$ is right permutative with respect to 
$\cT\cC(n,\cU)$, 
we must produce the dotted morphisms 
$p\colon d\to a$ in $\cS\cC(n,\cU)$ and 
$q\colon d\to c$ in $\cT\cC(n,\cU)$ 
in the commutative diagram below 
with $s\in\cS\cC(n,\cU)$ and $t\in\cT\cC(n,\cU)$: 
$$
\xymatrix{
d \ar@{-->}[r]^{\ \ \ \ q\in\cT\cC(n,\cU)} \ar@{-->}[d]_{p\in\cS\cC(n,\cU)} & 
c \ar[d]^{s\in\cS\cC(n,\cU)}\\
a \ar[r]_{t\in \cT\cC(n,\cU)} & b.
}$$
Then there are morphisms 
$p'\colon d'\to a_n$ in $\cS$ and 
$q'\colon d'\to c_n$ in $\cT$ 
such that $s_nq'=t_np'$. 
We construct the dotted morphisms 
$p_k''\colon d_k''\to a_k$ $(0\leq k\leq n-1)$ in $\cS$ 
in the commutative diagram below 
by using the assumption that 
$\cS$ is right permutative with respect to $\cU$ and 
descending induction on $k$:
$$
\xymatrix{
d_0'' \ar[r]^{i_0^{d''}} \ar@{-->}[d]_{p_0''} & 
d_1'' \ar[r]^{i_1^{d''}} \ar@{-->}[d]_{p_1''} & 
\cdots \ar[r]^{i_{n-2}^{d''}} & 
d_{n-1}'' \ar[r]^{i_{n-1}^{c''}} \ar@{-->}[d]_{p_{n-1}''} & 
d_n''=d' \ar[d]^{p_n''=p'}\\
a_0 \ar[r]_{i_0^a} & a_1 \ar[r]_{i_1^a} & \cdots \ar[r]_{i_{n-2}^a} & 
a_{n-1} \ar[r]_{i_{n-1}^a} & a_n.
}
$$
Hence we have the morphism 
$p''\colon d''\to a$ in $\cS\cC(n,\cU)$ such that 
$d_n''=d'$ and $p_n''=p'$. 
By applying the inductive hypothesis to 
$\iota^{\ast}(tp'')\colon \iota^{\ast}d''\to \iota^{\ast}b$ 
in $\cT\cC(n-1,\cU)$ 
and 
$\iota^{\ast}s\colon \iota^{\ast}c \to \iota^{\ast}b$ in $\cS\cC(n-1,\cU)$, 
we get the dotted morphisms 
$p^{(3)}\colon d^{(3)}\to i^{\ast}d''$ in $\cS\cC(n-1,\cU)$ 
and 
$q''\colon d^{(3)}\to i^{\ast}c$ 
in $\cT\cC(n-1,\cU)$ 
in the commutative diagram below:
$$
\xymatrix{
d^{(3)} \ar@{-->}[r]^{\ \ \ \ \ \ \ \ q''\in\cT\cC(n-1,\cU)} 
\ar@{-->}[d]_{p^{(3)}\in\cS\cC(n-1,\cU)} & 
\iota^{\ast}c \ar[d]^{\iota^{\ast}s\in\cS\cC(n-1,\cU)}\\
\iota^{\ast}d'' \ar[r]_{\ \ \ \ \ \ \ \ \iota^{\ast}(tp'')\in\cT\cC(n-1,\cU)} & 
\iota^{\ast}b.
}$$
By assumption $\cU\circ\cS\subset\cS\circ\cU$, 
there are morphisms 
$j\colon d^{(3)}_{n-1}\to d^{(4)}$ in $\cU$ and 
$p^{(4)}\colon d^{(4)}\to d'(=d_n'')$ in $\cS$ such that 
$p^{(4)}j=i_{n-1}^{d''}p_{n-1}^{(3)}$. 
We define $p\colon d\to a$ and $q\colon d\to c$ 
to be morphisms in $\cC(n,\cU)$ by setting 
as follows.
$$d_k:=
\begin{cases}
d_k^{(3)} & \text{if $0\leq k\leq n-1$}\\
d^{(4)} & \text{if $k=n$}
\end{cases},\ \ \ 
i_k^d:=
\begin{cases}
i_k^{d^{(3)}} & \text{if $0\leq k\leq n-2$}\\
j & \text{if $k=n$}
\end{cases},$$

$$
p_k:=
\begin{cases}
p_k'' p_k^{(3)} & \text{if $0\leq k\leq n-1$}\\
p' p^{(4)} & \text{if $k=n$}
\end{cases},\ \ \ 
q_k:=
\begin{cases}
q_k'' & \text{if $0\leq k\leq n-1$}\\
q' p^{(4)} & \text{if $k=n$}.
\end{cases}
$$
Notice that $s_n$ is a monomorphism by assumption and 
we have equalities 
\begin{multline*}
s_nq_ni_{n-1}^d = s_nq'p^{(4)}i_{n-1}^d=t_np'i_{n-1}^{d''}p_{n-1}^{(3)}
=t_ni_{n-1}^ap_{n-1}''p_{n-1}^{(3)}\\
=i_{n-1}^bt_{n-1}p_{n-1}''p_{n-1}^{(3)}
=i_{n-1}^bs_{n-1}q_{n-1}=s_ni_{n-1}^cq_{n-1}.
\end{multline*}
$$
\xymatrix{
d_{n-1} \ar[ddd]_{i_{n-1}^d} \ar[rrrr]^{q_{n-1}} \ar[rd]^{p_{n-1}^{(3)}} 
& & & & c_{n-1} \ar[ld]_{s_{n-1}} \ar[ddd]^{i_{n-1}^c}\\
& d_{n-1}'' \ar[r]^{p_{n-1}''} \ar[d]_{i_{n-1}^{d''}} 
& a_{n-1} \ar[r]^{t_{n-1}} \ar[d]^{i_{n-1}^a}
& b_{n-1} \ar[d]^{i^b_{n-1}} & \\
& d' \ar[r]^{p'} \ar[rrrd]^{q'} & a_n \ar[r]^{t_n} & b_n & \\
d_n \ar[rrrr]_{q_n} \ar[ru]_{p^{(4)}} & & & & c_n \ar[ul]_{s_n}.
}$$
Therefore we have the equality $q_ni_{n-1}^d=i_{n-1}^cq_{n-1}$. 
Namely $q\colon d\to c$ is actually a morphism in $\cC(n,\cU)$. 
By definition, 
$p$ is in $\cS\cC(n,\cU)$ and $q$ is in $\cT\cC(n,\cU)$ and we have $sq=tp$.

\sn
$\mathrm{(3)}$ 
In order to show that 
$\cS\cC(n,\cU)$ is right reversible with respect to $\cT\cC(n,\cU)$, 
for any morphisms 
$f$, $g\colon a\to b$ in $\cT\cC(n,\cU)$ 
such that there is a morphism $s\colon b\to c$ 
in $\cS\cC(n,\cU)$ such that $sf=sg$, 
we must produce a morphism $t\colon d\to a$ 
in $\cS\cC(n,\cU)$ such that $ft=gt$. 
Then there is a morphism $t'\colon d'\to a_n$ in $\cS$ such that 
$f_nt'=g_nt'$ by assumption. 
We construct the dotted morphisms 
$d_k''\colon d_k''\to a_k$ $(0\leq k\leq n-1)$ in $\cS$ 
in the commutative diagram below 
by using the assumption that 
$\cS$ is right permutative with respect to $\cU$ and 
descending induction on $k$:
$$
\xymatrix{
d_0'' \ar[r]^{i_0^{d''}} \ar@{-->}[d]_{t_0''} & 
d_1'' \ar[r]^{i_1^{d''}} \ar@{-->}[d]_{t_1''} & 
\cdots \ar[r]^{i_{n-2}^{d''}} & 
d_{n-1}'' \ar[r]^{i_{n-1}^{c''}} \ar@{-->}[d]_{t_{n-1}''} & 
d_n''=d' \ar[d]^{t_n''=t'}\\
a_0 \ar[r]_{i_0^a} & a_1 \ar[r]_{i_1^a} & \cdots \ar[r]_{i_{n-2}^a} & 
a_{n-1} \ar[r]_{i_{n-1}^a} & a_n.
}
$$
Hence there is a morphism $t''\colon d'\to a$ in $\cS\cC(n,\cU)$ such that 
$d_n''=d'$ and $t_n''=t'$. 
By applying the inductive hypothesis to 
$\iota^{\ast}(ft'')$, $\iota^{\ast}(gt'')\colon 
\iota^{\ast}d''\to \iota^{\ast}b$ and 
$\iota^{\ast}s\colon \iota^{\ast}b\to \iota^{\ast}c$, 
there is a morphism $t^{(3)}\colon d^{(3)}\to \iota^{\ast}d''$ in 
$\cS\cC(n-1,\cU)$ such that 
$\iota^{\ast}(ft'')t^{(3)}=\iota^{\ast}(st'')t^{(3)}$. 
Since $\cV\cC(n-1,\cU)$ is a right cofinal set in 
$\cS\cC(n-1,\cU)$ by $\mathrm{(1)}$, 
there is a morphism $t^{(4)}\colon d^{(4)}\to d^{(3)}$ 
in $\cV\cC(n-1,\cU)$ 
such that the composition $t^{(3)}t^{(4)}$ is in $\cV\cC(n-1,\cU)$. 
Then by the condition $\cU\circ\cV\subset\cV\circ\cU$, 
there are morphisms $j\colon d_{n-1}^{(4)}\to d^{(5)}$ in $\cU$ and 
$t^{(5)}\colon d^{(5)}\to d'$ in $\cS$ such that 
$t^{(5)}j=i^{d'}_{n-1}t_{n-1}^{(3)}t_{n-1}^{(4)}$. 
We define $t\colon d\to a$ to be a morphism in $\cS\cC(n;\cU)$ by setting 
as follows.
$$
d_k:=
\begin{cases}
d_k^{(4)} & \text{if $0\leq k\leq n-1$}\\
d^{(5)} & \text{if $k=n$}
\end{cases},
i_k^d:=
\begin{cases}
i_k^{d^{(4)}} & \text{if $0\leq k\leq n-2$}\\
j & \text{if $k=n-1$}
\end{cases},
t_k:=
\begin{cases}
t_k''t_k^{(3)}t_k^{(4)} & \text{if $0\leq k\leq n-1$}\\
t't^{(5)} & \text{if $k=n$}.
\end{cases}
$$
We can easily check that $ft=gt$. 
\end{proof}

\section{Fibration theorem revisited}
\label{sec:Fib thm}

In this section, let $\cC$ be a small category with cofibrations and 
we write $0$ and $\Cof\cC$ for the specific zero object and 
the set of all cofibrations in $\cC$ respectively. 
For any set of morphisms $u$ in $\cC$, 
we write $\cC^u$ for the full subcategory 
of those objects $x$ such that the canonical morphism $0\to x$ is in $u$. 
If a set of weak equivalences $w$ in $\cC$ is extensional, 
then $\cC^w$ is a subcategory with cofibrations in $\cC$. 
For any set of weak equivalences $w$ in $\cC$, 
we set $\bar{w}\colon =w\cap \Cof\cC$. 
Then the set $\bar{w}$ is strictly multiplicative. 
For any pair of sets of weak equivalences $v\subset w$ in $\cC$ such that $w$ 
satisfies the extension axiom, 
the inclusion functor $\cC^w\rinc \cC$ and the identity functor of $\cC$ 
induce the sequence 
\begin{equation}
\label{equ:fibration}
vS_{\cdot}\cC^w \to vS_{\cdot}\cC \to wS_{\cdot}\cC.
\end{equation}
The main objective of this section, we will give a sufficient condition 
that the sequence $\mathrm{(\ref{equ:fibration})}$ is a fibration up to homotopy. 
We start by looking into the proof of the generic fibration theorem in \cite{Wal85}. 

\begin{para}
\label{prop:char of fibration thm}
{\bf Proposition.} 
{\it
Let $v\subset w$ be sets of weak equivalences in $\cC$. 
Assume that $w$ is extensional. 
Then the sequence $\mathrm{(\ref{equ:fibration})}$ 
is a fibration up to homotopy 
if and only if the inclusion functor 
$\bar{w}S_{\cdot}\cC(-,v)\rinc wS_{\cdot}\cC(-,v)$ 
of bisimplicial categories 
is a homotopy equivalence.
}
\end{para}

\begin{proof}
Since $w$ is extensional, we have an isomorphism 
of bisimplicial categories
$$vS_{\cdot}S_{\cdot}(\cC^w\rinc \cC)\isoto\bar{w}S_{\cdot}\cC(-,v).$$ 
(See \cite[p.352]{Wal85}.) 
Let us consider the following commutative diagram: 
$$\xymatrix{
vS_{\cdot}\cC^w \ar[r] \ar@{=}[ddd] & vS_{\cdot}\cC \ar[r] \ar@{=}[ddd] & 
vS_{\cdot}S_{\cdot}(\cC^w\rinc \cC) \ar[d]_{\wr}\\
& & \bar{w}S_{\cdot}\cC(-,v) \ar[d]^{\textbf{I}}\\
& & wS_{\cdot}\cC(-,v)\\
vS_{\cdot}\cC^w \ar[r] & vS_{\cdot}\cC \ar[r] & wS_{\cdot}\cC \ar[u]_{\textbf{II}}.
}$$
Here the top line is a fibration sequence, up to homotopy and the map 
$\textbf{II}$ is a homotopy equivalence by \cite[1.5.7., 1.6.5.]{Wal85}. 
Hence the bottom line is a fibration sequence up to homotopy if and only if 
the map $\textbf{I}$ is a homotopy equivalence by Lemma~\ref{lem:fib seq} below.
\end{proof}

\begin{para}
\label{lem:fib seq}
{\bf Lemma.} 
{\it
Let 
\begin{equation}
\label{eq:fib diagram}
\xymatrix{
x \ar[r] \ar[d]_a & y \ar[r] \ar[d]_b & z \ar[d]^c\\
x' \ar[r] & y' \ar[r] & z'
}
\end{equation}
be a map of fibration sequences of topological spaces such that 
$z$ and $z'$ are connected. 
If $a$ and $b$ are weak equivalences, 
then $c$ is also a weak equivalence.
}
\end{para}

\begin{proof}
The diagram $\mathrm{(\ref{eq:fib diagram})}$ above induces 
a map of distinguished triangles in the stable category $\Sp$ of spectra. 
Then $a$ and $b$ induces isomorphisms in $\Sp$ and 
thus $c$ is also an isomorphism in $\Sp$ by the five lemma for 
distinguished triangles. 
Since $z$ and $z'$ are connected, $c$ is a weak equivalence.
\end{proof}

\begin{para}
\label{thm:fib thm}
{\bf Theorem. (Fibration theorem).} 
{\it
Let $v\subset w$ be sets of weak equivalences in $\cC$ such that 
$w$ is saturated, extensional and 
right localizing in $\cC$ and 
right permutative with respect to $v$ and 
$\bar{w}$ is right cofinal in $w$ and right permutative with respect to 
both $v$ and $\Cof\cC$ and $v\circ \bar{w}\subset \bar{w}\circ v$ and 
assume that all cofibrations in $\cC$ are monomorphisms. 
Then the sequence $\mathrm{(\ref{equ:fibration})}$ is a fibration up to homotopy.
}
\end{para}

\begin{para}
\label{rem:i_C}
{\bf Remark.} 
Let $w$ be the set of weak equivalences $w$ in $\cC$ 
and $v=i_{\cC}$ the set of all isomorphisms in $\cC$. 
Then we have $v\circ \bar{w}\subset \bar{w}\circ v$. 
Namley $v=i_{\cC}$ always satisfies the assumption in Theorem~\ref{thm:fib thm}. 
\end{para}

\begin{para}
\label{rem:not cover}
{\bf Remark.} 
Theorem~\ref{thm:fib thm} does not cover the original fibration theorem 
in \cite{Wal85}. 
For example let $\cC$ be the abelian category of bounded complexes 
over abelian groups and $w$ is the class of all quasi-isomorphisms in $\cC$. 
Then Waldhausen category $(\cC,w)$ does satisfy Waldhausen's conditions. 
But $\bar{w}$ is not right cofinal in $w$ as follows. 
Let $y=\bbZ/2$ concentrated in degree $0$, $x$ a projective 
resolution of $y$, 
say $2\colon \bbZ\to\bbZ$ and $x\to y$ the natural projection. 
The object $y$ has the property that any morphism $z \to y$ in $\bar{w}$ 
must be an isomorphism. 
So $y$ should split off $x$, and this cannot happen.
\end{para}

\begin{para}
\label{ex:exact fib}
{\bf Example.} 
Let $\cE$ be a small idempotent complete exact category, 
$\cA$ a right $s$-filtering subcategory of $\cE$ 
and $w$ a set of all weak isomorphisms associated to $\cA$ in $\cE$ 
in the sense of \cite{Sch04}. 
Then the sets $w$ and $v=i_{\cE}$ satisfy assuptions in Theorem~\ref{thm:fib thm} 
by \cite{Sch04}. 
Therefore we have the fibration sequence 
$$K(\cA)\to K(\cE)\to K(\cE;w).$$
\end{para}

\begin{proof}[Proof of Theorem~\ref{thm:fib thm}]
By proposition~\ref{prop:char of fibration thm}, 
we shall prove that 
the inclusion functor 
$\bar{w}S_{\cdot}\cC(-,v)\to wS_{\cdot}\cC(-,v)$ is a homotopy equivalence. 
To prove $\bar{w}S_{\cdot}\cC(-,v)\to wS_{\cdot}\cC(-,v)$ 
is a homotopy equivalence, 
the realization lemma in \cite[Appendix A]{Seg74} or \cite[5.1]{Wal78} 
reduces the problem 
to prove the inclusion functor 
$$\bar{w}S_{m}\cC(n,v)\to wS_{m}\cC(n,v)\ \ \ \ \ \ \ \mathrm{(\ast)}_{n,m}$$
is a homotopy equivalence for any non-negative integers $n$ and $m$. 
To prove the functor $ $ is a homotopy equivalence, 
we will utilzize Lemma~\ref{lem:fp com} $\mathrm{(4)}$ for $\cC=S_m\cC(n,v)$, 
$\cS=w$ and $\cT=\bar{w}$. 
Let $n$ be a non-negative integer 
and let $m$ be a positive integer. 
We set $\cB:=\cC(n,v)$ and $\cA:=\cB(m-1,\Cof\cB)$. 
Then since the forgetful functor gives 
an equivalence of categories with cofibrations 
$$S_m\cB\isoto\cA,$$ 
to check assumptions in Lemma~\ref{lem:fp com} $\mathrm{(4)}$, 
we shall show that $w\cA$ is saturated and right localizing in $\cA$ 
and $\bar{w}\cA$ is right cofinal in $w\cA$ by Lemma~\ref{lem:pullback}. 
We enumerate assumptions in Theorem~\ref{thm:fib thm}.\\
$\mathrm{(A)}$ 
$w$ is saturated and extensional.\\
$\mathrm{(B)}$ 
$w$ is right permutative with respect to $\Mor\cC$.\\
$\mathrm{(C)}$ 
$w$ is right permutative with respect to $v$.\\
$\mathrm{(D)}$ 
$w$ is right reversible with respect to $\Mor\cC$.\\
$\mathrm{(E)}$ 
$\bar{w}$ is right cofinal in $w$.\\
$\mathrm{(F)}$ 
$\bar{w}$ is right permutative with respect to $\Cof\cC$.\\
$\mathrm{(G)}$ 
$\bar{w}$ is right permutative with respect to $v$.\\
$\mathrm{(H)}$ 
$v\circ\bar{w}\subset \bar{w}\circ v$.\\
$\mathrm{(I)}$ 
All cofibrations in $\cC$ are monomorphims.

\begin{claim}
We have the following:\\
$\mathrm{(1)}$ 
$w\cA$ is extensional and saturated in $\cA$.\\
$\mathrm{(2)}$ 
$\bar{w}$ is right permutative with respect to $\Mor\cC$.\\
$\mathrm{(3)}$ 
$\bar{w}\cB$ is right cofinal in $w\cB$.\\
$\mathrm{(4)}$ 
$\bar{w}\cB$ is right permutative with respect to $\Mor\cB$.\\
$\mathrm{(5)}$ 
$\bar{w}\cB$ is right permutative with respect to $\Cof\cB$.\\
$\mathrm{(6)}$ 
$w\cB$ is right permutative with respect to $\Cof\cB$.\\
$\mathrm{(7)}$ 
$w\cB$ is right reversible with respect to $\Mor\cB$.\\
$\mathrm{(8)}$ 
$\bar{w}\cA$ is right cofinal in $w\cA$.\\
$\mathrm{(9)}$ 
$\bar{w}\cA$ is right permutative with respect to $\Mor\cA$.\\
$\mathrm{(10)}$ 
$w\cA$ is right permutative with respect to $\Mor\cA$.\\
$\mathrm{(11)}$  
$w\cA$ is right reversible with respect to $\Mor\cA$.
\end{claim}

\begin{proof}[Proof of claim]
Assertion $\mathrm{(1)}$ follows from $\mathrm{(A)}$ as in \cite{Wal85}.

\sn
$\mathrm{(2)}$ 
We apply Lemma~\ref{lem:fp com} $\mathrm{(1)}$ to 
$\cS=w$, $\cT=\bar{w}$ and $\cU=\Mor\cC$. 
Asumptions in Lemma~\ref{lem:fp com} $\mathrm{(1)}$ follow from 
$\mathrm{(B)}$, $\mathrm{(E)}$ and $\Mor\cC\circ \bar{w}\subset \Mor\cC$.

\sn
$\mathrm{(3)}$ 
We apply Lemma~\ref{lem:cC(-;U)} $\mathrm{(1)}$ to 
$\cS=w$, $\cT=\bar{w}$ and $\cU=v$. 
Assumptions in Lemma~\ref{lem:cC(-;U)} $\mathrm{(1)}$ follows from 
assumptions $\mathrm{(E)}$ and $\mathrm{(H)}$. 
Hence we get the result.

\sn
$\mathrm{(4)}$ 
We apply Lemma~\ref{lem:cC(-;U)} $\mathrm{(2)}$ to 
$\cS=\bar{w}$, $\cT=\Mor\cC$ and $\cU=v$. 
Assumptions in Lemma~\ref{lem:cC(-;U)} $\mathrm{(2)}$ follow from 
previously proved claim $\mathrm{(2)}$ and 
assumptions $\mathrm{(C)}$, $\mathrm{(H)}$ and $\mathrm{(I)}$ and 
$\Mor\cC\circ \bar{w}\subset \Mor\cC$. 
Hence we get the result. 

\sn
$\mathrm{(5)}$ 
We apply Lemma~\ref{lem:cC(-;U)} $\mathrm{(2)}$ to 
$\cS=\bar{w}$, $\cT=\Cof\cC$ and $\cU=v$. 
Assumptions in Lemma~\ref{lem:cC(-;U)} $\mathrm{(2)}$ follow from 
assumptions $\mathrm{(F)}$, $\mathrm{(G)}$, $\mathrm{(H)}$, $\mathrm{(I)}$ and 
$\Cof\cC\circ \bar{w}\subset \Cof\cC$. 
Hence we get the result. 

\sn
$\mathrm{(6)}$ 
We apply Lemma~\ref{lem:fp com} $\mathrm{(3)}$ to 
$\cS=w\cB$, $\cT=\bar{w}\cB$ and $\cU=\Cof\cB$. 
Asumptions in Lemma~\ref{lem:fp com} $\mathrm{(3)}$ follow from 
previously proved claims 
$\mathrm{(3)}$ and $\mathrm{(5)}$ and $\bar{w}\cB \circ \Cof\cB\subset \Cof\cB$.
Hence we obtain the result.

\sn
$\mathrm{(7)}$ 
We apply Lemma~\ref{lem:cC(-;U)} $\mathrm{(3)}$ to 
$\cS=\cV=w$, $\cT=\Mor\cC$ and $\cU=v$. 
Assumptions in Lemma~\ref{lem:cC(-;U)} $\mathrm{(3)}$ follow from 
assumptions $\mathrm{(A)}$, $\mathrm{(C)}$, $\mathrm{(D)}$ and 
$v\circ w\subset w\circ v$. 
The last condition follows from assumption $v\subset w$. 
Hence we get the result. 

\sn
$\mathrm{(8)}$ 
We apply Lemma~\ref{lem:cC(-;U)} $\mathrm{(1)}$ to 
$\cS=w\cB$, $\cT=\bar{w}\cB$ and $\cU=\Cof\cB$. 
Assumptions in Lemma~\ref{lem:cC(-;U)} $\mathrm{(1)}$ follow from 
previously proved claim 
$\mathrm{(6)}$ and $\bar{w}\cB \circ \Cof\cB\subset \bar{w}\cB\circ\Cof\cB$. 
The last condition follows from 
$\bar{w}\cB \circ \Cof\cB\subset \Cof\cB$.

\sn
$\mathrm{(9)}$ 
We apply Lemma~\ref{lem:cC(-;U)} $\mathrm{(2)}$ to 
$\cS=\bar{w}\cB$, $\cT=\Mor\cB$ and $\cU=\Cof\cB$. 
Assumptions in Lemma~\ref{lem:cC(-;U)} $\mathrm{(2)}$ follow from 
previosuly proved claim $\mathrm{(4)}$ and 
$\Cof\cC\circ \bar{w}\cB \subset \bar{w}\cB\circ\Cof\cC$ 
and $\Mor\cB\circ\bar{w}\cB\subset\Mor\cB$.

\sn
$\mathrm{(10)}$ 
We apply Lemma~\ref{lem:fp com} $\mathrm{(3)}$ to 
$\cS=w\cA$, $\cT=\bar{w}\cA$ and $\cU=\Mor\cA$. 
Assumptions in Lemma~\ref{lem:fp com} $\mathrm{(3)}$ follow from 
previously proved claims $\mathrm{(8)}$ and $\mathrm{(9)}$ and 
$\bar{w}\cA\circ \Mor\cA\subset \Mor\cA$.

\sn
$\mathrm{(11)}$ 
We apply Lemma~\ref{lem:cC(-;U)} $\mathrm{(3)}$ to 
$\cS=w\cB$, $\cT=\Mor\cB$ and $\cU=\Cof\cB$ and $\cV=\bar{w}\cB$. 
Assumptions in Lemma~\ref{lem:cC(-;U)} $\mathrm{(2)}$ follow from 
previously proved claims $\mathrm{(3)}$, $\mathrm{(6)}$ and $\mathrm{(7)}$ and 
the facts that 
$\Cof\cB\circ \bar{w}\cB\subset \bar{w}\cB\circ\Cof\cB$ and 
$\bar{w}\cB$ and $w\cB$ are multiplicative sets in $\cB$. 
\end{proof}

Hence the inclusion functor between bisimplicial categories 
$$\bar{w}S_{\cdot}\cC(-,v)\rinc wS_{\cdot}\cC(-,v)$$
is a homotopy equivalence. 
Then the sequence $\mathrm{(\ref{equ:fibration})}$ 
is a fibration sequence up to homotopy 
by Proposition~\ref{prop:char of fibration thm}. 
\end{proof}

\section{Localization of categories}
\label{sec:loc of cat}

In this section, 
we construct fibration sequences of $K$-theory of 
categories with cofibrations relating with localizations of 
categories. 
(See Corollary~\ref{cor:localization sequence}.) 
A key ingredient of the construction is 
a homotopy invariance of taking Gabriel-Zisman localization. 
(See Lemma~\ref{lem:localization of categories}.) 
By utilizing the same method, 
in the last of this section, 
we will prove a version of approximation theorem for Waldhausen $K$-theory. 
(See Proposition~\ref{prop:abstract approximation}.) 
We start by recalling the notations in localization theory of categories from 
\cite{GZ67}.

\begin{para}
\label{dfthm:category of fractions}
{\bf Definition-Theorem.} 
(\cf \cite[Chapter I]{GZ67}.)
Let $\cC$ be a small category and 
$\cS$ a right localizing set in $\cC$. 
Then\\
$\mathrm{(1)}$ 
We define $\cS^{-1}\cC$ to be a category by 
setting $\Ob\cS^{-1}\cC :=\Ob\cC$ and 
$$\displaystyle{\Hom_{\cS^{-1}\cC}(x,y):=
\underset{z\to x\in\cS}{\colim} \Hom(z,y)}$$
for any objects $x$ and $y$ in $\cC$. 
A morphism from an object $x$ to an object $y$ in 
$\cS^{-1}\cC$ is represented by a sequence of morphisms 
$x \overset{s}{\leftarrow} z \onto{f} y$ in $\cC$ with $s\in \cS$. 
We denote it by $f/s$. 
We also define 
$Q_{\cS}\colon\cC\to\cS^{-1}\cC$ to be a functor 
by sending an object $x$ to $x$ and a morphism $x\onto{f} y$ to 
the class $x \overset{\id_x}{\leftarrow} x \onto{f} y$.\\
$\mathrm{(2)}$ 
The pair $(\cS^{-1}\cC,Q_{\cS})$ satsfies the following universal property:\\
For any category $\cX$, the functor 
$$\cX^{Q_{\cS}}\colon\cX^{\cS^{-1}\cC}\to\cX^{\cC}$$
is an isomorphism from $\cX^{\cS^{-1}\cC}$ onto 
the full subcategory of $\cX^{\cC}$ 
whose objects are the functor 
$f\colon\cC\to\cX$ 
which makes all the morphisms of $\cS$ invertible. 
We say that $\cS^{-1}\cC$ is the {\it category of fractions of $\cC$ for $\cS$}, 
and $Q_{\cS}$ is the {\it canonical functor}.
\qed
\end{para}

\begin{para}
\label{nt:C^times}
For a category $\cC$, we write 
$\cC^{\times}$ for the maximum groupoid of $\cC$, 
Namely $\cC^{\times}$ is the category whose class of objects are 
same as $\cC$ and whose class of morphisms are all isomorphisms in $\cC$. 
\end{para}

\begin{para}
\label{lem:localization of categories}
{\bf Lemma.} 
{\it
Let $\cC$ be a small category and $w$ is a right localizing set in $\cC$. 
Then\\
$\mathrm{(1)}$ 
The canonical functor $Q_w\colon \cC \to w^{-1}\cC$ is a homotopy equivalence.\\
$\mathrm{(2)}$ 
Moreover assume that $w$ is saturated in $\cC$, 
then $Q_w$ induces an equivalence of categories 
$w^{-1}w\isoto {(w^{-1}\cC)}^{\times}$. 
}
\end{para}

\begin{proof}
$\mathrm{(1)}$ 
We will apply Quillen's Theorem A in \cite{Qui73} 
to the functor $Q_w$. 
Let $x$ be an object in $w^{-1}\cC$. 
We write $x/Q_w$ for the category consisting of pairs 
$(y,\alpha)$ with $\alpha\colon x\to y$ morphisms in $w^{-1}\cC$, 
in which a morphism from $(y,\alpha) $ to $(y',\alpha')$ is 
a morphism $u\colon y\to y'$ such that $Q_w(u)\alpha=\alpha'$. 
We will show that $x/Q_w$ is cofiltered. 
Namely we will check the following three conditions:

\sn
$\mathrm{(a)}$ 
$x/Q_w$ is non-empty.

\sn
$\mathrm{(b)}$ 
For every two objects $(y,\alpha) $ and $(y',\alpha')$ in $x/Q_w$, 
there exists an object $(y'',\alpha'')$ and two morphisms 
$u\colon(y'',\alpha'')\to (y,\alpha)$ and $v\colon(y'',\alpha'')\to (y,\alpha)$.

\sn
$\mathrm{(c)}$ 
For every two parallel morphisms 
$u,\ v\colon(y,\alpha)\to (y',\alpha')$ in $x/Q_w$, 
there exists an object $(y'',\alpha'')$ and a morphism 
$j\colon(y'',\alpha'')\to(y,\alpha)$ such that $uj=vj$.

Since $(x,\id_x/\id_x)$ is in $x/Q_w$, 
$x/Q_w$ is non-empty. 
Let $(y,\alpha/s)$ and $(y',\alpha'/s')$ with 
$x\ \overset{s}{\leftarrow} z \onto{\alpha} y$ and 
$x' \overset{s'}{\leftarrow} z' \onto{\alpha'} y'$ 
be a pair of objects in $x/Q_w$. 
We will show condition $\mathrm{(b)}$. 
By right permutative condition, 
there exists morphisms $t'\colon z''\to z$ and $t\colon z''\to z'$ 
such that $s't=st'$. 
Then we can check that $(z'',st'/\id_{z''}) $ is in $x/Q_w$ and 
there exists morphisms 
$(y',\alpha'/s') 
\overset{\alpha' t}{\leftarrow} 
(z'',st'/\id_{z''}) 
\onto{\alpha t'} (y,\alpha/s)$. 
Next we will show condition $\mathrm{(c)}$. 
For any two parallel morphisms 
$u,\ v\colon(y,\alpha/s)\to (y',\alpha'/s')$, 
by using right permutative condition, we can find 
morphisms $t\colon z''\to z$ and $t'\colon z''\to z'$ 
such that $st=s't'$, $u\alpha t=\alpha't'=v\alpha t$. 
Hence there exists a morphism 
$\alpha t\colon(z'',st/\id_{z''})\to (y,\alpha/s)$ 
such that $u\alpha t=v\alpha t$. 
Therefore $x/Q_w$ is contractible by Corollary~2 (of Proposition~3) 
in \cite[\S 1]{Qui73}. 
Thus by Theorem~A in \cite{Qui73} again, 
$Q_w$ is a homotopy equivalence.

\sn
$\mathrm{(2)}$ 
Obviously the functor $w^{-1}w\to {\left (w^{-1}\cC \right )}^{\times}$ 
is essentially surjective, 
what we need to show is that the functor is fully faithful. 
Let $x$ and $y$ be a pair of objects in $\cC$ 
and $x \overset{s}{\leftarrow} z \onto{f} y$ 
a morphism in $w^{-1}\cC$. 
We will show that $f/s$ is an isomorphism 
if and only if $f$ is in $w$. 
If $f$ is in $w$, 
we can check that $s/f$ is the inverse morphism of $f/s$. 
If $f/s$ is an isomorphism, 
then there exists the inverse morphism 
$y \overset{t}{\leftarrow} z' \onto{g} x$ of $f/s$. 
Since $f/s \cdot g/t=\id_x$, 
there exists morphisms $s'\colon z''\to z'$ in $w$ and 
$g'\colon z''\to z$ such that $sg'=gs'$ and $ts'=fg'$. 
By saturation condition, it turns out that $f$ is in $w$. 
Thus the functor $w^{-1}w\to {\left (w^{-1}\cC \right )}^{\times} $ is full. 
Next let 
$x\overset{s}{\leftarrow} z \onto{f} y$ and 
$x'\overset{s'}{\leftarrow} z' \onto{f'} y'$ 
be a pair of morphisms in $w^{-1}w$. 
If the equality $ $ holds in $w^{-1}\cC$. 
Then there exists morphisms $t\colon z''\to z$ and $t'\colon z''\to z'$ 
such that 
we have equalities 
$s't'=st$ and $f't'=ft$ and $st$ is in $w$. 
Therefore by saturation condition, $t$ and $t'$ are also in $w$ 
and the equality $f/s=f'/s'$ holds in $w^{-1}w$. 
Thus the functor $w^{-1}w\to {\left (w^{-1}\cC \right )}^{\times}$ 
is faithful. 
Hence we complete the proof.
\end{proof}

\begin{para}
\label{cor:comparison of K(C;w) and K(w^-1C)}
{\bf Corollary.} 
{\it
Let $(\cC,w)$ be a small Waldhausen category. 
Assume that $w$ is a saturated, right localizing set and 
$w^{-1}\cC$ is also a category with cofibrations such that the canonical 
functor $Q_w\colon \cC \to w^{-1}\cC$ is exact and reflects exactness 
and for any non-negative integer $n$, $Q_w$ induces an equivalence of 
categories $w^{-1}S_n\cC\isoto S_nw^{-1}\cC$. 
Then $Q_w$ induces a homotopy equivalence $wS_{\cdot}\cC\to iS_{\cdot}w^{-1}\cC$. }
\end{para}

\begin{proof}
For any non-negative integer $n$, 
let us consider the commutative diagram of the canonical functors 
$$
\xymatrix{
wS_n\cC \ar[dr]_{\text{I}} \ar[rr] & & iS_nw^{-1}\cC\\
& {(w^{-1}S_n\cC)}^{\times}.\ar[ru]_{\text{II}} &\\
}
$$
Here the functor $\textbf{I}$ is a homotopy equivalence and 
the functor $\textbf{II}$ is an equivalence of categories 
by the previous lemma~\ref{lem:localization of categories}. 
Hence the canonical functor induced from $Q_w$ is a homotopy equivalence 
$wS_{\cdot}\cC\to iS_{\cdot}w^{-1}\cC $ 
by the realization lemma in \cite[Appendix A]{Seg74} or \cite[5.1]{Wal78}.
\end{proof}

\begin{para}
\label{cor:localization sequence}
{\bf Corollary.} 
{\it
Let $(\cC,w)$ be a small Waldhausen category. 
We set $\bar{w}:=w\cap \Cof\cC$. 
We assume the following conditions:\\
$\mathrm{(i)}$ 
$w$ is saturated.\\
$\mathrm{(ii)}$ 
$w\cS_n\cC$ is right localizing in $\cS_n\cC$ for any 
non-negative integer $n$.\\
$\mathrm{(iii)}$ 
$w^{-1}\cC$ is also a category with cofibrations such that the 
canonical functor 
$Q_w\colon \cC\to w^{-1}\cC$ 
is exact and reflects exactness.\\
$\mathrm{(iv)}$ 
For any non-negative integer $n$, 
$Q_w$ induces an equivalence of categories 
$w^{-1}\cS_n\cC\isoto\cS_nw^{-1}\cC$.\\
$\mathrm{(v)}$ 
$\bar{w}$ is right cofinal in $w$.\\
$\mathrm{(vi)}$ 
$\bar{w}$ is right permutative with respect to $\Cof\cC$.\\
$\mathrm{(vii)}$ 
All cofibrations in $\cC$ are monomorphisms.\\
Then the inclusion functor 
$\cC^w\rinc \cC$ and the canonical localization functor 
$Q_w\colon\cC \to w^{-1}\cC$ induces a fibration sequence up to homotopy
$$i\cS_{\cdot}\cC^w\to i\cS_{\cdot}\cC\to i\cS_{\cdot}w^{-1}\cC.$$
}
\end{para}

\begin{proof}
Let us consider the commutative diagram below:
$$
\xymatrix{
i\cS_{\cdot}\cC^w\ar[r]  & 
i\cS_{\cdot}\cC \ar[r] \ar[dr]_{i\cS_{\cdot}Q_w} & w\cS_{\cdot}\cC 
\ar[d]^{\textbf{I}}\\
& & i\cS_{\cdot}w^{-1}\cC.
}$$
Then the top line is a fibration sequence up to 
homotopy by 
Theorem~\ref{thm:fib thm} 
and the map $\textbf{I}$ 
is a homotpy equivalence by 
Corollary~\ref{cor:comparison of K(C;w) and K(w^-1C)}. 
Hence we obtain the result.
\end{proof}

By using the similar method as in the proof of 
Corollary~\ref{cor:comparison of K(C;w) and K(w^-1C)}, 
we prove a version of an approximation theorem for Waldhausen $K$-theory 
in the following way.

\begin{para}
\label{prop:abstract approximation}
{\bf Proposition (Approximation theorem).} 
{\it
Let $f\colon (\calD,v)\to (\cC,w)$ be an exact functor 
between Waldhausen categories. 
Suppose that for any non-negative integer $n$, 
$wS_n\cC$ and $vS_n\calD$ are right localizing in 
$S_n\cC$ and $S_n\calD$ respectively and 
$f$ induces an equivalence of categories 
$v^{-1}S_n\calD\isoto w^{-1}S_n\cC$. 
Then $f$ induces a homotopy equivalence 
$vS_{\cdot}\calD \to wS_{\cdot}\cC$.
}
\end{para}

\begin{proof}
The realization lemma in \cite[Appendix A]{Seg74} or \cite[5.1]{Wal78} reduces the problem to prove that 
for any non-negative integer $n$, 
$vS_n\calD \to wS_n\cC$ the induced morphism by $f$ is a homotopy equivalence. 
We fix a non-negative integer $n$ and let us consider the commutative 
diagram below.
$$
\xymatrix{
vS_n\calD \ar[r] \ar[d] & wS_n\cC \ar[d]\\
{\left (v^{-1} S_n\calD \right )}^{\times} \ar[r]_{\sim} & 
{\left (w^{-1}S_n\cC \right )}^{\times}. 
}
$$
Here the bottom line is an equivalence of categories, a fortiori, 
a homotopy equivalence and the vertical morphisms are homotopy equivalences 
by Lemma~\ref{lem:localization of categories}. 
Hence the inclusion functor in 
$vS_n\calD \to wS_n\cC$ is a homotopy equivalence. 
We complete the proof. 
\end{proof}

\section{Localization of diagrams of modules}
\label{sec:Loc diag of mod}

Let $A$ be a noetherian commutative ring with $1$ and 
$\fS$ a multiplicative closed set of $A$. 
Let $\cM_A$ be the category of 
finitely generated $A$-modules. 
We denote a skelton of $\cM_A$ by the same letters $\cM_A$ 
and we shall assume that $\cM_A$ is a small category. 
In this subsection, we will study the category of 
diagrams in $\cM_{\fS^{-1}A}$.

\begin{para}
\label{para:diagram category}
Recall from Conventions that 
for a small category $\cI$ and a category $\cX$, 
we write $\cX^{\cI}$ for the category of functors 
from $\cI$ to $\cX$ whose morphisms are natural transformations. 
We sometimes call a functor $\cI \to \cX$ an {\it $\cI$-diagram in $\cX$}. 
For a small category $\cI$, the category $\cM_A^{\cI}$ is enriched 
over the category of $A$-modules. Namely 
for any objects $x$ and $y$ in $\cM_A^{\cI}$, 
$\Hom_{\cM_A^{\cI}}(x,y)$ naturally becomes $A$-modules. 
Indeed, for any $f$, $g\colon x\to y$ in 
$\Hom_{\cM_A^{\cI}}(x,y)$ and $a$ in $A$, 
we set 
$${(f+g)}_i(t):=f_i(t)+g_i(t),$$
$${(af)}_i(t):=af_i(t)$$
for any object $i$ of $\cI$ and any element $t$ of $x$. 
Moreover the $A$-modules structures on $\Hom$-sets are compatible with 
the composition of morphisms. 
$\cM_A^{\cI}$ is also an abelian category. 
\end{para}

\begin{para}
\label{para:exact localization functor}
Let $\cI$ be a small category. 
The base change functor 
$-\otimes_A \fS^{-1}A\colon \cM_A \to \cM_{\fS^{-1}A}$ 
which sends an $A$-module $M$ to $\fS^{-1}M$ 
induces an exact functor 
$$\calL_{\cI,A,\fS}\colon \cM_A^{\cI}\to \cM_{\fS^{-1}A}^{\cI}.$$ 
For simplicity, we sometimes write $\calL$ for $\calL_{\cI,A,\fS}$. 
\end{para}

In the rest of this subsection, 
let $\{\cI_i\}_{1\leq i\leq r}$ be a family of finite 
totally ordered sets. 
We set $\displaystyle{\cJ:=\prod_{i=1}^r\cI_i}$.

\begin{para}
\label{prop:Hom set of localized cat}
{\bf Proposition.} 
{\it
For any objects $x$ and $y$ in $\cM_A^{\cJ}$, 
there is a canonical isomorphisms of $\fS^{-1}A$-modules
\begin{equation}
\label{eq:local hom}
\Hom_{\cM_{\fS^{-1}A}^{\cJ}}(\calL(x),\calL(y))\isoto \fS^{-1}\Hom_{\cM_A^{\cJ}}(x,y).
\end{equation}
}
\end{para}

\begin{proof}
We proceed by induction on $r$. 
If $\cJ$ is $[0]$ the totally ordered set $\{0\}$, 
then $\cM_A^{\cJ}$ and $\cM_{\fS^{-1}A}^{\cJ}$ are 
$\cM_A$ and $\cM_{\fS^{-1}A}$ respectively and 
assertion follows from Proposition~19 in \cite[Chapter II \S 2.7]{Bou61}. 
We assume that assertion is true for $r=r'$. 
We will prove assertion for $r=r'+1$. 
We set $\cC:=\cM_A^{\prod_{i=1}^{r'}\cI_i}$ and 
$\cC':=\cM_{\fS^{-1}A}^{\prod_{i=1}^{r'}\cI_i}$. 
Without loss of generality, we shall assume that 
$\cI_{r'+1}$ is $[n]$ the finite totally ordered set 
of non-negative integres less than $n$ 
with the usual ordering for some positive integer $n$. 

Then there is a 
description of $\Hom$-sets of $\cC^{[n]}$ in terms of 
suitable kernels of finite direct sum of $\Hom$-sets in $\cC$ as follows. 
Let $x$ and $y$ be objects in $\cC^{[n]}$ and let $i$ 
be a non-negative integer less than $n-1$. 
We define 
$D_{x,y,i}\colon\Hom_{\cC}(x_i,y_i)\times\Hom_{\cC}(x_{i+1},y_{i+1})\to 
\Hom_{\cC}(x_i,y_{i+1})$ 
to be a homomorphism of $A$-modules 
by sending a pair $(a\colon x_i\to y_i,b\colon x_{i+1}\to y_{i+1})$ to 
$b x(i\leq i+1)-y(i\leq i+1)a$. 
Then the family $\{D_{x,y,i}\}_{1\leq i\leq n-1}$ induces a 
homomorphism of $A$-modules
$$D_{x,y}\colon\bigoplus_{i\in[n]}\Hom_{\cC}(x_i,y_i)\to 
\bigoplus_{i\in[n-1]} \Hom_{\cC}(x_i,y_{i+1})$$
and there is a canonical isomorphism of $A$-modules
\begin{equation}
\label{eq:Ker description}
\Hom_{\cC^{[n]}}(x,y)\isoto \Ker D_{x,y}.
\end{equation}
Since the base change functor 
$-\otimes_A\fS^{-1}A\colon \cM_A\to \cM_{\fS^{-1}A}$ 
is exact, 
there are isomorphisms of $\fS^{-1}A$-modules 
\begin{eqnarray*}
\fS^{-1}\Hom_{\cC^{[n]}}(x,y) & \isoto & 
\Ker \left ( \bigoplus_{i\in[n]}\fS^{-1}\Hom_{\cC}(x_i,y_i)\to 
\bigoplus_{i\in[n-1]} \fS^{-1}\Hom_{\cC}(x_i,y_{i+1}) \right )\\
& \underset{\textbf{I}}{\isoto} & 
\Ker \left ( \bigoplus_{i\in[n]}\Hom_{\cC'}(\calL(x_i),\calL(y_i))\to 
\bigoplus_{i\in[n-1]} \Hom_{\cC'}(\calL(x_i),\calL(y_{i+1})) \right )\\
& \underset{\textbf{II}}{\isoto} & \Ker D_{\calL(x),\calL(y)}\\
& = & \Hom_{\cC'^{[n]}}(\calL(x),\calL(y)).
\end{eqnarray*}
Here the isomorphisms \textbf{I} and \textbf{II} 
come from inductive hypothesis and 
the isomorphism $\mathrm{(\ref{eq:Ker description})}$ 
for $\fS^{-1}A$ respectively. 
\end{proof}

\begin{para}
\label{defcor:localization of diagram}
{\bf Definition-Corollary.} 
{\it
We consider the canonical exact functor 
$\calL_{\ast}\colon\cM_A^{\cJ}\to \cM_{\fS^{-1}A}^{\cJ}$ 
induced from the base change functor 
$-\otimes_A\fS^{-1}A\colon\cM_A\to\cM_{\fS^{-1}A}$. Then\\
$\mathrm{(1)}$ 
For morphisms $f$, $g\colon x\to y$ in $\cM_A^{\cJ}$, 
$\calL_{\ast}(f)=\calL_{\ast}(g)$ 
if and only if there is an element $s$ in $\fS$ such that 
$sf=sg$.\\
$\mathrm{(2)}$ 
Let $\cC$ be a full subcategory of $\cM_A^{\cJ}$. 
We write $w:=\Isom_{\fS,\cC}$ for 
the class of all morphisms $f$ in $\cC$ such that 
$\calL_{\ast}(f)$ are isomorphisms. 
Then $\Isom_{\fS,\cC}$ is a saturated strictly multiplicative set.\\
$\mathrm{(3)}$ 
A morphism $f\colon x\to y$ in $\cC$ is in 
$\Isom_{\fS,\cC}$ if and only if 
there are a morphism $g\colon y\to x$ and elements $s$, $t$ and $u$ in $\fS$ 
such that $fgt=st\id_y$ and $gfu=su\id_x$.\\
$\mathrm{(4)}$ 
Let $f\colon x\to y$ be a morphism in $\Isom_{\fS,\cC}$ and 
$g\colon z\to y$ a morphism in $\cC$. 
Then there are a morphism $h\colon z\to x$ and an element $s$ of $\fS$ 
such that $sg=fh$.\\
$\mathrm{(5)}$ 
$w$ is a right  localizing system in $\cC$.\\
$\mathrm{(6)}$ 
For any pair of objects $x$ and $y$ in $\cC$, 
a morphism from $x$ to $y$ in $w^{-1}\cC$ is 
represented by $f/s\id_x$ where $f$ is a morphism from 
$x$ to $y$ and $s$ is an element of $\fS$.\\
$\mathrm{(7)}$ 
The induced functor $w^{-1}\cC\to\cM_{\fS^{-1}A}^{\cJ}$ from $\calL_{\ast}$ 
is fully faithful.\\
$\mathrm{(8)}$ 
Assume that $\cC$ is a category with cofibrations such that 
the inclusion functor $\cC \rinc \cM_A^{\cJ}$ is exact and reflects 
exactness, then $\Isom_{\fS,\cC} $ satisfies the guling axiom.\\
$\mathrm{(9)}$ 
Moreover assume that $w^{-1}\cC$ is a category with cofibrations 
such that the induced functor 
$w^{-1}\cC\to \cM_{\fS^{-1}A}^{\cJ}$ from $\calL_{\ast}$ 
is exact and reflects exactness. 
Let $n$ be a non-negative integer. 
We regard $F_n\cC$ as a full subcategory of $\cM_A^{[n]\times\cJ}$. 
Then we have the equality $wF_n\cC=\Isom_{\fS,F_n\cC}$.\\
$\mathrm{(10)}$ 
We assume same assumptions as in $\mathrm{(9)}$. 
Then for any non-negative integer the canonical functor 
$w^{-1}S_n\cC \to S_nw^{-1}\cC$ 
induced from the functor $Q_w\colon\cC \to w^{-1}\cC$ 
is fully faithful.\\
$\mathrm{(11)}$ 
Let $n$ be a non-negative integer and assume that 
the condition:

\sn
For any object $x$ in $S_nw^{-1}\cC$, 
there are an object $y$ in $S_n\cC$ and a morphism $s\colon y\to x$ in $w$.

\sn
Then the canonical functor $w^{-1}S_n\cC\to S_nw^{-1}\cC$ 
induced from $Q_w$ is an equivalence of categories.\\
$\mathrm{(12)}$ 
If the condition in $\mathrm{(11)}$ holds for any non-negative integer $n$, 
then $Q_w$ induces a homotopy equivalence 
$wS_{\cdot}\cC\to iS_{\cdot}w^{-1}\cC$ on $K$-theory.\\
$\mathrm{(13)}$ 
Assume that $\cC$ is closed under taking kernels 
of the morphisms in $\cC$ which are 
epimorphisms in $\cM_A^{\cJ}$. 
Then for any pair of composable morphisms 
$x\onto{f} y\onto{g}  z$ in $\cC$, 
if $gf$ and $f$ are cofibrations in $\cC$, then 
$g$ is also a cofibration in $\cC$.\\
$\mathrm{(14)}$ 
Assume that the condition in $\mathrm{(13)}$ and 
moreover assume that for any element $s$ of $\fS$ and 
for any object $x$ in $\cC$, 
$\coim(s\colon x\to x)$ and 
$\coker(s\colon x\to x)$ 
are in $\cC$. 
Then $\bar{w}:=w\cap\Cof\cC$ is right permutative with 
respect to $\Cof \cC$.
}
\end{para}

\begin{proof}
$\mathrm{(1)}$ 
By Proposition~\ref{prop:Hom set of localized cat}, 
the equality 
$f/1=g/1$ holds in 
$\Hom_{\cM_{\fS^{-1}A}^{\cJ}}(\calL_{\ast}(x),\calL_{\ast}(y))$ 
if and only if there is an element $s$ in $\fS$ 
such that $(f\cdot 1 -g\cdot 1)s=0$. 
Thus we obtain the result.

\sn
$\mathrm{(2)}$ 
Let $x\onto{f}y\onto{g} z$ be a pair of composable morphisms in $\cM_A^{\cJ}$. 
If $f$ is an isomorphism, then $\calL_{\ast}(f)$ is also an isomorphism. 
Therefore $f$ is in $w$. 
Thus $w$ contains all isomorphisms in $\cM_A^{\cJ}$. 
Next if two of $\calL_{\ast}(f)$, $\calL_{\ast}(g) $ and 
$\calL_{\ast}(gf)$ are isomorphisms, 
then the third one is also an isomorphism. 
Hence $w$ is a saturated set. 

\sn
$\mathrm{(3)}$ 
If $f$ is in $w$. 
Then there is a morphism $g/s\colon y\to x$ in $\cM_{\fS^{-1}A}^{\cJ}$ 
such that $fg/s=\id_y/1$ and $gf/s=\id_x/1$ in $\cM_{\fS^{-1}A}^{\cJ}$. 
These equalities mean that there are elements $t$ and $u$ in $\fS$ 
such that $(fg-s\id_y)t=0$ and $(gf-s\id_x)u=0$. 
Conversely if there are a morphism $g\colon y\to x$ and elements 
$s$, $t$ and $u$ in $\fS$ such that 
$fgt=st\id_y$ and $gfu=su\id_x$, 
then 
$f/1\cdot g/s=\id_y/1$ and 
$g/s\cdot f/1=\id_x/1$ in $\cM_{\fS^{-1}A}^{\cJ}$.

\sn
$\mathrm{(4)}$ 
By $\mathrm{(3)}$, 
there are a morphism $h'\colon y\to x$ in $\cC$ 
and elements $s'$, $t$ and $u$ in $\fS$ such that 
$fh't=s't\id_y$ and $h'fu=s'u\id_x$. 
We set $h:=h'gt$ and $s:=s't$. 
Then we have equalities 
$fh=fh'tg=s'tg=sg$. 

\sn
$\mathrm{(5)}$ 
By $\mathrm{(2)}$ and $\mathrm{(4)}$, 
what we need to prove is that $w$ is right permutative. 
First we will show right permutative condition for $w$. 
Namely for any pair of morphisms $f,\ g\colon x\to y$ 
in $\cC$ if there exists a morphism $s\colon y\to z$ in $w$ 
such that $sf=sg$, then we need to produce a morphism $t\colon u\to x$ in $w$ 
such that $ft=gt$. 
Indeed, in this case we have equalities
$$\calL_{\ast}(f)={\calL(s)}^{-1}\calL_{\ast}(sf)=
{\calL(s)}^{-1}\calL_{\ast}(sg)=\calL_{\ast}(g).$$
Hence by $\mathrm{(1)}$, 
there is an element $t$ of $\fS$ such that  
$ft\id_x=tf=tg=gt\id_x$. 
Thus we complete the proof.

\sn
$\mathrm{(6)}$ 
Let $x$ and $y$ be a pair of objects in $\cC$. 
A morphism from $x$ to $y$ in $w^{-1}\cC$ is represented by 
$f/t$ where $f\colon z\to y$ is a morphism in $\cC$ and 
$t\colon z\to x$ is a morphism in $w$ by 
Definition-Theorem~\ref{dfthm:category of fractions}. 
Then by $\mathrm{(3)}$, there is a morphism $g\colon x\to z$ 
and an element $s$ of $\fS$ such that $tg=s\id_x$. 
Therefore we have the equality $f/t=fg/s\id_x$ in $w^{-1}\cC$. 
Thus we complete the proof.

\sn
$\mathrm{(7)}$ 
Let $x$ and $y$ be a pair of objects of $\cC$. 
By virtue of Proposition~\ref{prop:Hom set of localized cat}, 
a morphism $\calL_{\ast}(x)$ to $\calL_{\ast}(y)$ in $\cM^{\cJ}_{\fS^{-1}A}$ 
is represented by $f/s$ where 
$f$ is a morphism from $x$ to $y$ in $\cC$ and 
$s$ is an element of $\fS$. 
Hence we have the equality $\calL_{\ast}(f/s\id_x)=f/s$ 
and the functor 
$w^{-1}\cC\to\cM^{\cJ}_{\fS^{-1}A}$ is full. 
Next we consider a pair of morphisms from $x$ to $y$ in $w^{-1}\cC$. 
By $\mathrm{(6)}$, they are represented by $f/s\id_x$ and $f'/s'\id_x$ 
where $f,\ f'\colon x\to y$ are morphisms in $\cC$ and 
$s$ and $s'$ are elements in $\fS$. 
Assume that we have the equality 
$\calL_{\ast}(f/s\id_x)=\calL_{\ast}(f'/s'\id_x)$. 
Then there is an element $t$ of $\fS$ such that 
$ts'f=tsf'$. 
Then we have the equalities 
$f/s\id_x=ts'f/sts'\id_x=tsf'/sts'\id_x=f'/s'\id_x$. 
Thus the functor $w^{-1}\cC\to \cM_{\fS^{-1}A}^{\cJ}$ induced from 
$\calL_{\ast}$ is faithful.

\sn
$\mathrm{(8)}$ 
Let us consider the commutative diagram in $\cC$ below
$$\xymatrix{
x \ar[d]_a & z \ar@{>->}[l] \ar[r] \ar[d]_b & y \ar[d]^c\\
x' & z' \ar@{>->}[l] \ar[r]  & y'
}$$
where $a$, $b$ and $c$ are morphisms in $\Isom_{\fS,\cC}$. 
What we need to prove is the induced morphism 
$a\sqcup_b c\colon x\sqcup_z y \to x'\sqcup_{y'}z'$ 
is also in $\Isom_{\fS,\cC}$. 
By applying the functor $\calL_{\ast}$ to the diagram above, 
it turns out that $\calL_{\ast}(a\sqcup_b c)=
\calL_{\ast}(a)\sqcup_{\calL_{\ast}(b)}\calL_{\ast}(c)$ 
is an isomorphism by exactness of $\calL_{\ast}$. 
Thus $a\sqcup_b c\colon x\sqcup_z y \to x'\sqcup_{y'}z'$ 
is in $\Isom_{\fS,\cC}$. 

\sn
$\mathrm{(9)}$ 
Let $a\colon x\to y$ be a morphism in $wF_n\cC$. 
Then by $\mathrm{(3)}$, 
for each integer $0\leq k\leq n$, 
there are a morphism $b_k\colon y_k\to x_k$ and elements 
$s_k$, $t_k$ and $u_k$ in $\fS$ such that 
$a_kb_kt_k=s_kt_k\id_{y_k}$ and $b_ka_ku_k=\id_{x_k}$. 
We set $\displaystyle{b_k':=
\left (\prod_{\substack{i=0 \\ i\neq k}}^{n}s_i \right )
\left (\prod_{i=0}^n (u_it_i)\right )b_k }$ and 
$\displaystyle{c:=\prod_{i=0}^{n}(s_it_iu_i)}$. 
Then we have the equalities $a_kb_k'=c\id_{y_k}$ and 
$b_k'a_k=c\id_{x_k}$. 
Therefore for each $0\leq k\leq n-1$, 
we have the equalities
$$i_k^xb_k'a_k=ci_k^x=b'_{k+1}a_{k+1}i^x_k=b'_{k+1}i_{k}^ya_k.$$
Moreover since $\calL_{\ast}(a_k)$ is an isomorphism in 
$\cM_{\fS^{-1}A}^{\cJ}$, 
the diagram below is commutative:
$$
\xymatrix{
\calL_{\ast}(y_k) \ar[r]^{\calL_{\ast}(i_k^y)} \ar[d]_{\calL_{\ast}(b'_k)} & 
\calL_{\ast}(y_{k+1}) \ar[d]^{\calL_{\ast}(b'_{k+1})}\\
\calL_{\ast}(x_k) \ar[r]_{\calL_{\ast}(i_k^x)} & 
\calL_{\ast}(x_{k+1}).
}
$$
Hence by $\mathrm{(1)}$, there is an element $v_k$ in $\fS$ 
such that $v_kb'_{k+1}i_k^y=i_k^xv_kb'_k$. 
We set $\displaystyle{v:=\prod_{i=0}^{n-1}v_i}$ 
and for each integer $0\leq k\leq n$, 
we set $b''_k:=vb'_k$. 
Then for any integer $k$, we have the equalities $i^x_kb''_k=b''_{k+1}i^y_k$. 
Namely $b''$ is a morphism from $y$ to $x$ in $F_n\cC$ and we 
have the equalities $ab''=cv\id_y$ and $b''a=cv\id_x$. 
By applying $\mathrm{(3)}$ to the morphisms $a$ and $b''$ and the element 
$cv$ of $\fS$, 
it turns out that $a$ is in $\Isom_{\fS,\cC} $. 
Conversely we can show that a morphism in $\Isom_{\fS,F_n\cC}$ is contained in 
$wF_n\cC$ by utilizing $\mathrm{(3)}$ again. 
We complete the proof.

\sn
$\mathrm{(11)}$ 
Obviously assumption says that the canonical functor 
$w^{-1}S_n\cC\to S_nw^{-1}\cC$ is essentially surjective. 
Then by $\mathrm{(10)}$, it turns out that 
the functor is an equivalence of categories $w^{-1}S_n\cC\isoto S_nw^{-1}\cC $. 

\sn
$\mathrm{(12)}$ 
By $\mathrm{(11)}$, 
for any non-negative integer $n$, the canonical functor 
$w^{-1}S_n\cC\to S_nw^{-1}\cC$ is an equivalence of categories. 
Hence by Corollary~\ref{cor:comparison of K(C;w) and K(w^-1C)}, 
we obtain the result. 

\sn
$\mathrm{(13)}$ 
First notice that $g$ is a cofibration in $\cM_A^{\cJ}$. 
Let us consider the commutative diagram in $\cM_A^{\cJ}$ below:
$$\xymatrix{
x \ar@{>->}[r] \ar@{=}[d] & y \ar@{->>}[r] \ar@{>->}[d] & y/x \ar@{>->}[d]\\
x \ar@{>->}[r] \ar@{>->}[d] & z \ar@{->>}[r] \ar@{=}[d] & z/x \ar@{->>}[d]\\
y \ar@{>->}[r] & z \ar@{->>}[r] & z/y.
}$$
Here the sequence $x\rinf y\rdef y/x$ 
is a cofibration sequence in $\cM_A^{\cJ}$. 
On the other hand, by assumption, 
$y/x\isoto \ker(z/x\rdef z/y)$ is isomorphic to an object in $\cC$. 
Since the inclusion functor $\cC\rinc \cM_A^{\cJ}$ reflects exactness, 
$g$ is a cofibration in $\cC$.

\sn
$\mathrm{(14)}$ 
What we need to show is that for a morphism 
$f\colon x\to y$ in $\bar{w}$ and a cofibration 
$g\colon z\to y$ in $\cC$, there exists an object $u$ in $\cC$ 
and a morphism $s'\colon u \to z$ in $\bar{w}$ and a cofibration 
$h'\colon u\to x$ in $\cC$ 
such that $gs'=fh'$. 
In this case, 
by $\mathrm{(4)}$, 
there are a morphism $h\colon z\to x$ and an element $s$ of $\fS$ 
such that $gs=fh$. 
Then since $f$ and $g$ are monomorphisms in $\cC$, 
we have equalities $\ker s=\ker gs=\ker fh =\ker h$. 
Thus the morphisms $s\colon z\to z$ and $h\colon z\to x$ 
induce the morphisms $\bar{s}\colon z/\ker s \to z$ 
and $\bar{h}\colon z/\ker s \to x$ which makes the diagram below 
commutative:
$$
\xymatrix{
z/\ker s \ar[r]^{\bar{s}} \ar[d]_{\bar{h}} & z \ar[d]^g\\
x\ar[r]_f & y.
}
$$
By assumption, the sequence $z/\ker s\rinf z \rdef \coker s$ 
is a cofibration sequence in $\cC$. 
Hence $\bar{s}$ is a morphism in $\bar{w}$. 
Then by the equality $g\bar{s}=f\bar{h}$ and $\mathrm{(13)}$, 
$\bar{h}$ is also a cofibration in $\cC$. 
We complete the proof.
\end{proof}

\begin{para}
\label{cor:fibration seq for diagram of modules}
{\bf Corollary.} 
{\it
Let $\cC$ be a full subcategory of $\cM_A^{\cJ}$ and 
we write $w$ for $\Isom_{\fS,\cC}$. 
We assume the following conditions:\\
$\mathrm{(1)}$ 
$\cC$ and $w^{-1}\cC$ are catogories with cofibrations 
such that the inclusion functor 
$\cC\rinc \cM_A^{\cJ}$ and the canonical functor 
$w^{-1}\cC \to \cM_{\fS^{-1}A}^{\cJ}$ induced from 
the base change functor 
$-\otimes_A\fS^{-1}A\colon\cC 
\to \cM_{\fS^{-1}A}^{\cJ}$ are 
exact and reflect exactness.\\
$\mathrm{(2)}$
$\cC$ is closed under taking kernels of 
the morphisms in $\cC$ which are epimorphisms in $\cM_A^{\cJ}$.\\
$\mathrm{(3)}$ 
For any element $s$ of $\fS$ and 
any object $x$ in $\cC$, 
the objects $\coim(s\colon x\to x) $ and 
$\coker(s\colon x\to x)$ are also in $\cC$.\\
$\mathrm{(4)}$ 
For any non-negative integer $n$, and any object $x$ in $S_nw^{-1}\cC$, 
there exists an object $y$ in $S_n\cC$ and a morphism $y\to x$ in $w$.\\
$\mathrm{(5)}$ 
The set of trivail cofibrations 
$\bar{w}:=w\cap \Cof\cC$ is right cofinal in $w$.

\sn
Then the inclusion functor $\cC^w\rinc \cC$ 
and the canonical functor $\cC \to w^{-1}\cC$ induce 
a fibration sequence up to homotopy
$$iS_{\cdot}\cC^w\to iS_{\cdot}\cC\to iS_{\cdot}w^{-1}\cC.$$
}
\end{para}

\begin{proof}
We will apply Corollary~\ref{cor:localization sequence} to the sequence 
$iS_{\cdot}\cC^w\to iS_{\cdot}\cC\to iS_{\cdot}w^{-1}\cC$. 
The assumptions 
in the Corollary follow from 
Definition-Corollary~\ref{defcor:localization of diagram}.
\end{proof}

Finally we give an application of 
Proposition~\ref{prop:abstract approximation} to 
categories of diagrams of modules.

\begin{para}
\label{cor:approximation of diagram of modules}
{\bf Corollary.} 
{\it
Let $\cC$ be a full subcategory of $\cM_A^{\cJ}$ and let 
$\calD$ be a full subcategory of $\cC$. 
We write $w_{\cC}$ and $w_{\calD}$ for 
$\Isom_{\fS,\cC}$ and $\Isom_{\fS,\calD}$ respectively. 
Assume that $\cC$ and $\calD$ are categories with cofibrations 
such that the inculsion functors $\calD\rinc\cC \rinc \cM_A^{\cJ}$ 
are exact and reflect exactness. 
Suppose that the inclusion functor $\calD\rinc \cC$ satisfies 
the following two conditions.\\
$\mathrm{(1)}$ 
For any object $x$ in $\cC$, there are an object $y$ in $\calD$ 
and a morphism $h\colon y\to x$ in $w_{\cC}$.\\
$\mathrm{(2)}$ 
For any cofibration $i\colon x\rinf y$ in $\cC$ and any 
morphism $a\colon y' \to y$ in $w_{\cC}$ with $y'$ in $\calD$, 
there are a cofibration $i'\colon x'\to y'$ in $\calD$ 
and a morphism $b\colon x'\to x$ in $w_{\cC}$ such that $ib=ai'$.
$$
\xymatrix{
x' \ar@{-->}[r]^{b\in w_{\cC}} \ar@{>-->}[d]_{i'} & x \ar@{>->}[d]^{i}\\
y' \ar[r]_{a\in w_{\cC}} & y. 
}
$$
Then the inclusion functor $\calD\rinc \cC$ induces a homotopy equivalence 
$w_{\calD}S_{\cdot}\calD \to w_{\cC}S_{\cdot}\cC$. 
}
\end{para}

\begin{proof}
We will aplly Proposition~\ref{prop:abstract approximation} 
to the inclusion functor 
$(\calD,w_{\calD}) \to (\cC,w_{\cC})$. 
For simplicity we write $w$ for both $w_{\cC}$ and $w_{\calD}$. 
Let $n$ be a non-negative integer. 
Since $S_n\cC$ is equivalent to $F_{n-1}\cC$ and 
$F_{n-1}\cC$ is a full subcategory of $\cC^{[n-1]}$, 
$wS_n\cC$ is right localizing in $S_n\cC$ 
by Definition-Corollary~\ref{defcor:localization of diagram} $\mathrm{(5)}$. 
Similarly $wS_n\calD$ is right localizing in $S_n\calD$. 
We will show that the inclusion functor 
$wS_n\calD \rinc wS_n\cC$ 
induces an equivalence 
of categories $\iota_{\ast}\colon w^{-1}S_n\calD\isoto w^{-1}S_n\cC$. 
To prove that $\iota_{\ast}$ is essentially surjective, 
we claim the following assertion. 

\sn
{\bf Claim.} 
For any object $x$ in $F_n\cC$ and a morphism $s\colon y\to x_n$ in $w_{\cC}$, 
there is an object $x'$ in $F_n\calD$ and a morphism 
$a\colon x'\to x$ in $wF_n\cC$ such that $x'_n=y$ and $a_n=s$. 

\begin{proof}[Proof of claim]
We proceed by induction on $n$. 
If $n=1$, then assertion is trivial, and if $n=2$, assertion is just 
assumption $\mathrm{(2)}$. 
Suppose that assertion is true for $n=n'$ and we will consider 
assertion for $n=n'+1$. 
Let $x$ be an object in $F_{n'+1}\cC$ and 
let $s\colon y\to x_{n'+1}$ be a morphims in $w\cC$. 
Then by assumption $\mathrm{(2)}$, 
there are a cofibration $i'\colon y'\rinf y$ 
and a morphism $t\colon y'\to x_{n'}$ in $w$ such that 
$i_{n'}^xt=si'$. 
By inductive hypothesis, there is an object $x''$ in $F_{n'}\calD$ 
and a morphism $u\colon x'' \to j_{\ast}x$ in 
$wF_{n'}\cC$ such that 
$x''_{n'}=y'$ and $u_{n'}=t$. 
Then we define $x'$ and $a'\colon x'\to x$ to be 
an object in $F_{n'+1}\calD$ and a morphism in $wF_{n'+1}\calD$ as follows. 
$$
x'_k=
\begin{cases}
x''_k & \text{if $k\leq n' $}\\
y & \text{if $k=n'+1$}
\end{cases},\ \ 
i_k^{x'}=
\begin{cases}
i_k^{x''} & \text{if $k\leq n'-1 $}\\
i' & \text{if $k=n'$}
\end{cases},\ \ 
a_k=
\begin{cases}
u_k & \text{if $k\leq n'$}\\
s & \text{if $k=n'+1$}.
\end{cases}
$$
We can prove that $x'$ and $a$ are what we want. 
Thus we complete the proof. 
\end{proof}
In particular, it turns out that 
$\iota_{\ast}$ is essentially surjective. 
To prove $\iota_{\ast}$ 
is fully faithful, we will utilize Lemma~9.1 in \cite{Kel96}. 
To apply the lemma, what we need to prove is that 
for any morphism $s\colon x\to x'$ in $w_{\cC}$ with $x'$ in $\calD$, 
there are an object $x''$ in $\calD$ and a morphism $m\colon x''\to x$ 
such that $sm\colon x'' \to x'$ is in $w_{\calD}$. 
This condition follows from 
Definition-Corollary~\ref{defcor:localization of diagram} $\mathrm{(3)}$. 
Hence the functor $\iota_{\ast}\colon w^{-1}S_n\calD\isoto w^{-1}S_n\cC$ 
is an equivalence of categories. 
Then by Proposition~\ref{prop:abstract approximation}, 
the inclusion functor $\calD\rinc \cC$ induces a homotopy equivalence 
$wS_{\cdot}\calD \to wS_{\cdot}\cC$.
\end{proof}

\mn
SATOSHI MOCHIZUKI\\
{\it{DEPARTMENT OF MATHEMATICS,
CHUO UNIVERSITY,
BUNKYO-KU, TOKYO, JAPAN.}}\\
e-mail: {\tt{mochi@gug.math.chuo-u.ac.jp}}\\

\end{document}